\documentclass{amsart}
\usepackage[utf8]{inputenc}    
\usepackage[english]{babel}     
\usepackage[a4paper, left=1.5in, right=1.5in, top=1.5in, bottom=1.5in]{geometry}
\usepackage{graphicx}    
\usepackage{amsmath,amssymb}   
\usepackage{amsthm,mathrsfs}          
\usepackage{mathtools}        
\usepackage{enumitem}           
\usepackage{listings}           
\usepackage{todonotes}          
\usepackage{hyperref}        
\usepackage{tikz-cd}
\usepackage{amscd}
\usepackage{color}
\usepackage[all]{xy}

\numberwithin{equation}{section}

\title{Duality for Arithmetic $p$-adic Pro-\'etale Cohomology of Analytic Spaces}
\author{Zhenghui Li}

\newcommand{\Addresses}{{
  \bigskip
  \footnotesize

\textsc{IMJ-PRG, Sorbonne Universit\'e, 4 place Jussieu, 75005 Paris, France}\par\nopagebreak
  \textit{E-mail address}: \texttt{zhenghuili@imj-prg.fr}\par\nopagebreak

}}

\newtheorem{theorem}{Theorem}[section]
\newtheorem{corollary}[theorem]{Corollary}

\newtheorem{lemma}[theorem]{Lemma}
\newtheorem{proposition}[theorem]{Proposition}
\theoremstyle{definition}
\newtheorem{example}[theorem]{Examples}
\newtheorem{definition}[theorem]{Definition}
\newtheorem{notation}[theorem]{Notation}
\newtheorem{remark}[theorem]{Remark}

\def\bN{{\mathbb N}}
\def\bZ{{\mathbb Z}}
\def\bQ{{\mathbb Q}}

\def\bC{{\mathbb C}}
\def\bF{{\mathbb F}}

\def\cO{{\mathcal O}}
\def\cV{{\mathcal V}}

\def\cW{{\mathcal W}}
\def\bD{{\mathbb D}}

\def\cM{{\mathcal M}}

\def\#{{\sharp}}

\def\cE{{\mathcal{E}}}
\def\cF{{\mathcal{F}}}

\def\cG{{\mathcal{G}}}
\def\sD{{\mathscr{D}}}
\def\sC{{\mathscr{C}}}
\def\sF{{\mathscr{F}}}

\def\sH{{\mathscr{H}}}

\def\cB{{\mathcal{B}}}

\DeclareMathOperator{\proet}{pro\acute{e}t}
\DeclareMathOperator{\et}{\acute{e}t}

\DeclareMathOperator{\colim}{colim}
\DeclareMathOperator{\HK}{HK}
\DeclareMathOperator{\dR}{dR}
\DeclareMathOperator{\st}{st}
\DeclareMathOperator{\syn}{syn}
\DeclareMathOperator{\Hom}{Hom}
\DeclareMathOperator{\Mod}{Mod}

\DeclareMathOperator{\Spa}{Spa}

\DeclareMathOperator{\im}{Im}
\DeclareMathOperator{\coim}{Coim}

\DeclareMathOperator{\coker}{coker}

\DeclareMathOperator{\qcoh}{QCoh}
\DeclareMathOperator{\perf}{Perf}
\DeclareMathOperator{\nuc}{Nuc}

\DeclareMathOperator{\solid}{Solid}
\DeclareMathOperator{\cond}{Cond}
\DeclareMathOperator{\Ext}{Ext}
\DeclareMathOperator{\ssolid}{solid}
\DeclareMathOperator{\scond}{cond}
\DeclareMathOperator{\FF}{FF}
\DeclareMathOperator{\FFalg}{FF^{alg}}
\DeclareMathOperator{\alg}{alg}
\DeclareMathOperator{\cont}{cont}
\DeclareMathOperator{\crys}{crys}

\DeclareMathOperator{\tr}{tr}

\DeclareMathOperator{\Ab}{Ab}

\newcommand{\un}[1]{\underline{#1}}

\begin{document}

\begin{abstract}
	Let $K$ be a finite extension of $\bQ_p$. We prove that the  arithmetic $p$-adic pro-\'etale cohomology of smooth partially  proper spaces over $K$ satisfies a duality, as conjectured by Colmez, Gilles and Nizio{\l} in \cite{curves}. We derive it from the geometric duality on the Fargues-Fontaine curve from \cite{geometricdual} by Galois descent techniques of Fontaine.
\end{abstract}
\maketitle
 
\tableofcontents

\section{Introduction}
We prove in this  paper an arithmetic duality theorem for $p$-adic pro-\'etale cohomology  of smooth partially proper spaces over local fields of mixed  characteristic. 
\subsection{Motivation} The reader  may ask: why would one  be interested   in such  a duality~? For a guiding answer let us look at the algebraic geometric setting.   

 Let $K$ be a finite extension of $\bQ_p$ and $G_K$ be its absolute Galois group. Choose an algebraic closure $\overline{K}$ of $\bQ_p$. 
Arithmetic duality for $p$-adic \'etale cohomology of algebraic varieties over $K$ is a classical subject with many applications. In dimension zero, where the space is basically $\mathrm{Spec}(K)$, it reduces to  the local Tate duality -- the cohomological duality underlying local class field theory.   In higher dimensions, arithmetic dualities  relate Picard and Brauer groups of $p$-adic curves and   appear in the study of   the Albanese map, zero cycles, and $p$-adic cycle-class maps. They also provide  the cohomological framework for $p$-adic heights of algebraic cycles: Via \'etale Abel–Jacobi maps, complementary cycles give classes in dual Galois cohomology groups and the height pairing arises as a  refinement of arithmetic duality. 

    Our interest  in arithmetic dualities in the rigid analytic setting comes mainly from the wish to  import  some of the above mentioned applications. Since many rigid analytic varieties of arithmetic interest, such as local Shimura varieties, are partially proper, we focus on such varieties. We also learned that pro-\'etale cohomology behaves much better than \'etale cohomology on such spaces\footnote{For example, for the open unit disk $\bD_C^1$, $H^2_{\et}(\bD_C^1,\bQ_p)$ is non-trivial. So there is no Artin-Grothendieck vanishing for $p$-adic \'etale cohomology on Stein spaces, but $p$-adic pro-\'etale cohomology has this property.}, we focus on the former.

\subsection{Statement of the main theorem} 
 Let $X$ be a smooth partially proper rigid analytic variety over $K$, geometrically irreducible of dimension $d$. The arithmetic trace map\footnote{The compactly supported cohomology we use here is defined in \cite{achinger2025compactlysupportedpadicproetale}. The trace map is the geometric trace  (\cite[Sec. 7.3]{achinger2025compactlysupportedpadicproetale}) composed with the Galois trace map. See Section \ref{sectiontrace} for a recall.} 
	\begin{equation}\label{formulatrace}
	    \tr: H^{2d+2}_{\proet,c}(X,\bQ_p(d+1))\longrightarrow \bQ_p
	\end{equation}
	yields a pairing in the derived category of solid $\bQ_p$-vector spaces $D(\bQ_{p,\square})$
	\begin{equation}\label{pairingforarithmetic}
	R\Gamma_{\proet}(X,\bQ_p(j))\otimes^{L_\square}_{\bQ_p} R\Gamma_{\proet,c}(X,\bQ_p(d+1-j))[2d+2]\longrightarrow \bQ_p.\end{equation}
  The  main theorem  of this paper is  the following duality statement  conjectured in \cite[Conjecture 1.6]{curves}:
 \begin{theorem}\label{conjecturemain}\rm(Arithmetic duality)
\begin{enumerate}[leftmargin=*]
	\item The pairing \eqref{pairingforarithmetic} induces a duality quasi-isomorphism in $D(\bQ_{p,\square})$
		\begin{equation} \label{formuladesired1}
		\gamma_{\proet}: R\Gamma_{\proet}(X,\bQ_p(j))\stackrel{\simeq}\longrightarrow \bD_{\bQ_p}(R\Gamma_{\proet,c}(X,\bQ_p(d+1-j))[2d+2]),
		\end{equation}
		where $\bD_{\bQ_p}(-):=R\Hom_{\bQ_p}(-,\bQ_p)$ is the (derived) dual in $D(\bQ_{p,\square})$.

		\item Assume further that   $X$ is a Stein space. Then:
		\begin{enumerate}
		\item  The cohomology groups $H^i_{\proet}(X,\bQ_p(j))$ and $H^i_{\proet,c}(X,\bQ_p(j))$ are nuclear Fr\'echet and of compact type, respectively. 
		\item The pairing \eqref{pairingforarithmetic} induces isomorphisms of solid $\bQ_p$-vector spaces
			\begin{align}\label{kolo1}
		H^i_{\proet}(X,\bQ_p(j)) & \simeq H^{2d+2-i}_{\proet,c}(X,\bQ_p(d+1-j))^*,\\
		H^i_{\proet,c}(X,\bQ_p(j)) & \simeq H^{2d+2-i}_{\proet}(X,\bQ_p(d+1-j))^*\notag
		\end{align}
		where $(-)^*$ is the condensed  dual. 
		\end{enumerate}
	\end{enumerate}
\end{theorem}
\begin{example} Let $C$ be the $p$-adic completion of $\overline{K}$.  In his pioneering  work \cite{scholze2013adic} Scholze proved finiteness of $p$-adic \'etale cohomology for smooth $proper$ rigid analytic varieties over $C$ and raised the question of establishing Poincar\'e duality for them. This question was affirmatively answered by Lan-Liu-Zhu (\cite{lan2023rham}) with rational coefficients and by Gabber, Zavyalov (\cite{zavyalov2025mod}) and Mann (\cite{mann2022p}) with integral coefficients. Having these results, the arithmetic duality of Theorem \ref{conjecturemain} for smooth $proper$ varieties follows by Galois descent just as in the algebraic setting (see \cite[Corollary. 8.3]{curves} for an argument). 
\end{example}

\begin{example}
Contrary to the smooth proper case, the cohomology groups $H^i_{\proet}(X,\bQ_p)$ are in general infinite dimensional. Let us look at the example of  $X=\bD_K^1$, the open unit disc in dimension $1$, which also   shows why arithmetic duality should hold and points in the direction of  a possible proof. 

  We have (see \cite[Example. 1.7]{curves}): 
  \begin{align*}
H^1_{\proet}(X,\bQ_p(1)) & \simeq \mathcal{O}(D)/K\oplus H^1(G_K,\bQ_p(1)),\\
 H^3_{\proet,c}(X,\bQ_p(1)) & \simeq \mathcal{O}(\partial D)/\mathcal{O}(D)\oplus H^1(G_K,\bQ_p).
\end{align*}
Here   $\partial D:=\varprojlim_n D\setminus D_n$, where  $D_n$ is the closed ball $\{v_p(z)\geq \frac{1}{n}\}$, 
is the ``ghost circle'', 
the boundary of the unit disc.
The duality  in Theorem \ref{conjecturemain} is induced by Galois duality and coherent duality (i.e., Serre duality):
\begin{align*}
H^i(G_K,\bQ_p) & \simeq H^{2-i}(G_K,\bQ_p(1))^*,\\
H^0(D,\Omega^1_D) & \simeq H^1_c(D,\mathcal{O}_D)^*.
\end{align*}
We used here the isomorphisms:
$$
H^1_c(D,\mathcal{O}_D)\stackrel{\sim}{\leftarrow} \mathcal{O}(\partial D)/\mathcal{O}(D),\quad H^0(D,\Omega^1_D)\stackrel{\sim}{\leftarrow} \mathcal{O}(D)/K,
$$
where the first one is basically the
definition of $H^1_c(D,\mathcal{O}_D)$ since 
$H^1(D,\mathcal{O}_D)=0$; the second is induced by $f\mapsto df$: it is an isomorphism
since $H^1_{\rm dR}(D)=0$. 
Since $[K:\bQ_p]<\infty$, the above coherent $K$-duality can be turned into a $\bQ_p$-duality by
composing with ${\rm Tr}_{K/\bQ_p}$.
\end{example}
\begin{remark}
    \begin{enumerate}
        \item We do not expect to be able to exchange the places  of $R\Gamma_{\proet}$ and $R\Gamma_{\proet,c}$ in the derived duality \eqref{formuladesired1}. We found it  surprising that this can be done on  the level of cohomology group   as in \eqref{kolo1}.
        \item A general solid $\bQ_p$-vector space $V$ can have a nonvanishing $\Ext^n(V,\bQ_p)$ for large $n$. Thus to  deduce the duality on the cohomology groups as in    \eqref{kolo1} one needs a nontrivial functional analytic input. 
    \end{enumerate}
\end{remark}

\begin{remark}
From Theorem \ref{conjecturemain}, we can  deduce a duality for arithmetic $p$-adic pro-\'etale cohomology groups of smooth dagger affinoid varieties (cf. Corollary \ref{corollarydaggeraffinoid}).
 A derived duality similar to \eqref{formuladesired1} (one needs to exchange the positions of $R\Gamma_{\proet}$ and $R\Gamma_{\proet,c}$) for  smooth quasi-compact dagger varieties  should also be true and could be proved by combining the arguments in  \cite{geometricdual} and this paper. 
\end{remark}

\subsection{Outline of the proof} 
To prove Theorem \ref{conjecturemain}, it was tempting to try the familiar path: use the geometric Poincar\'e duality and apply Galois descent. However the geometric duality (as proved in \cite{geometricdual}, \cite{anschutz20246functorformalismsolidquasicoherent})  is stated in categories of (pre)sheaves on perfectoid affinoids and it is not clear how to pass, reasonably fully-faithfully,  from Galois equivariant versions of such objects to condensed sets. 

 Instead we decided to start with a version of geometric  duality on the Fargues-Fontaine curve from \cite{geometricdual} (in fact only depends on \cite{geometricdualold}) and do Galois descent from there. 
 
\subsubsection{Derived version}  Since partially proper varieties are locally Stein one can reduce, by a limit-colimit argument, to Stein spaces. Let $*\in \{\emptyset,c\}$ (compactly supported or not). We quote:

	\begin{theorem}[Duality on the Fargues-Fontaine curve, {\cite[Theorem 5.20]{geometricdual}}]  \label{kolo4} \ \\
 	Let $X$ be a smooth connected Stein space over $K$ of dimension $d$.  Then there is a natural quasi-isomorphism of  $G_K$-equivariant quasi-coherent sheaves on the Fargues-Fontaine curve $\mathrm{FF}$ 
 	\begin{equation}\label{kolo3}
 		\cE_{\proet}(X_C,\bQ_p)\cong R\sH om_{\FF}(\cE_{\proet,c}(X_C,\bQ_p(d))[2d],\cO),
 	\end{equation}
    where $\cE_{\proet,*}(X_C,\bQ_p)$ is a  sheaf  representing the (compactly supported) pro-\'etale cohomology. 
 \end{theorem}
 By definition, we have $$R\Gamma(\mathrm{FF}, \cE_{\proet,*}(X_C,\bQ_p)\simeq R\Gamma_{\proet,*}(X_C,\bQ_p).$$
Recall that Theorem \ref{kolo4}  is proved by lifting  the syntomic comparison theorem to quasi-coherent sheaves on the Fargues-Fontaine curve and then using  the duality of filtered de Rham and Hyodo-Kato cohomologies  for Stein spaces. 

  More precisely, assume $r\geq 2d$. Then, via the syntomic comparison, we have a distinguished triangle of quasi-coherent sheaves
\begin{equation}\label{kolo2}
\cE_{\proet,*}(X_C,\bQ_p(r)) \to \cE_{\mathrm{HK},*}(X_C,r) \to \cE_{\mathrm{dR},*}(X_C,r),
\end{equation}
where $\cE_{\mathrm{HK},*}(X_C,r),  \cE_{\mathrm{dR},*}(X_C,r)$ are  sheaves representing twisted Hyodo-Kato cohomology and filtered $\mathbf{B}^+_{\mathrm{dR}}$-cohomology quotients:
\begin{align*}
R\Gamma(\mathrm{FF},\cE_{\mathrm{HK},*}(X_C,r)) & \simeq  [R\Gamma_{\mathrm{HK},*}(X_C)\otimes^{\mathrm{L}_{\Box}}_{\breve{C}}\mathbf{B}^+_{\mathrm{st}}]^{N=0,\phi=p^r},\\
 R\Gamma(\mathrm{FF},\cE_{\mathrm{dR},*}(X_C,r)) & \simeq R\Gamma_*(X_C/\mathbf{B}^+_{\mathrm{dR}})/F^r,
\end{align*}
where $R\Gamma_*(X_C/\mathbf{B}^+_{\mathrm{dR}})\simeq R\Gamma_{\dR,*}(X)\otimes^{\mathrm{L}_{\Box}}_K\mathbf{B}^+_{\mathrm{dR}}$. 	This reduces the proof of  duality for the sheaves $\cE_{\proet,*}(X_C,\bQ_p(r)) $ to  the one for $\cE_{\mathrm{HK},*}(X_C,r)$ and $\cE_{\mathrm{dR},*}(X_C,r)$, which, in turn, follows from known Hyodo-Kato and 
filtered de Rham dualities for Stein spaces. 

	 To descend the geometric  duality \eqref{kolo3} to the arithmetic duality \eqref{formuladesired1}, we prove the following theorem, which we call ``Local Tate duality for the pro-\'etale sheaves"

\begin{theorem}[Tate duality, {Theorem \ref{theoremlocaltatesheaf1}}]
	Let $X$ be a smooth geometrically connected Stein space over $K$ and $r\geq 2d$. There is a natural quasi-isomorphism in $D(\bQ_{p,\square}):$
 	$$
 		R\Gamma(G_K,R\Hom_{\FF}(\cE_{\mathrm{\proet},c}(X_C,\bQ_p(r)),\cO_{\FF}\otimes \ \bQ_p(1)))\simeq \bD_{\bQ_p}(R\Gamma(G_K,R\Gamma_{\proet,c}(X_C,\bQ_p(r)[2]))
 	$$
\end{theorem}

The proof again evokes  the syntomic comparison theorem as in the geometric case  to reduce to analogous statements for  the sheaves representing Hyodo-Kato and de Rham cohomologies.  For the Hyodo-Kato part, we may assume that the complex representing Hyodo-Kato cohomology is perfect 
(locally Hyodo-Kato cohomology is of finite rank) and then Tate duality for perfect complexes on the Fargues-Fontaine curve reduces to Fontaine's duality for  almost $C$-representations  (\cite{katopresque},\cite{fontaine2019almost}).  For the de Rham part, the Galois action is simple although the cohomology groups are huge, and we prove it via an explicit computation using the fact that the complex $\cE_{\mathrm{dR},*}(X_C,r)$ is supported at the point at infinity of the Fargues-Fontaine curve. 

   Duality \eqref{formuladesired1}  follows from the above two theorems.

\subsubsection{Nonderived version}   
To show the topological properties of the cohomology groups of Stein spaces as  stated in the Theorem \ref{conjecturemain}(2),  we use the Hochschild-Serre spectral sequence
$$E_2^{a,b}=H^a(G_K,H^b_{\proet,*}(X_C,\bQ_p(j)))\Rightarrow H^{a+b}_{\proet,*}(X,\bQ_p(j))$$
where $*\in \{\emptyset,c\}$. We study this spectral sequence via applying syntomic comparison theorems. The spectral sequence degenerates on the third page and, locally on $X$, we show that most of the terms are finite dimensional $\bQ_p$-vector spaces and the other terms admit the desired property. However, globalization  for compactly supported pro-\'etale cohomology is  subtle and we  finish the proof only after showing the derived duality \eqref{formuladesired1}. 

  Finally,  duality \eqref{kolo1}  follows from the derived duality  and  reflexivity of cohomology groups.
  
\begin{remark} In the algebraic setting, arithmetic dualities are derived, via Galois descent, from geometric dualities. In the analytic setting treated here, however, it was the arithmetic dualities that were investigated first by Colmez-Gilles-Nizio{\l} in their initial computations. This is because they were expected to be simple Poincar\'e dualities hence amenable to explicit calculations  that seemed  easier than the expected calculations for Verdier dualities in the geometric setting. And, in fact, the first proof of analytic arithmetic dualities (for analytic curves) was done in \cite{curves} in a direct way, without descending from geometric dualities. 
\end{remark}
\begin{remark}[Related Work] 
\begin{enumerate}[leftmargin=*]
\item  Theorem \ref{conjecturemain} was known for curves (see \cite{curves}). 
The proof  in the one-dimensional case  is reduced to wide open curves  then deduced from an explicit computation on the 'ghost circle'. In higher dimensional case, this, of course, would not work. 
\item This paper was  written before the paper \cite{anschutz20246functorformalismsolidquasicoherent} of  Ansch\"utz, Le Bras and Mann 
appeared. In Remark 6.2.3 in loc. cit. the authors state that the derived part  of our arithmetic duality theorem, i.e., the quasi-isomorphism  \eqref{formuladesired1}, can be obtained  as a byproduct of their work on a six functor formalism for solid quasi-coherent sheaves on the Fargues-Fontaine curve.
\end{enumerate}
\end{remark}

\subsection*{Organization of the paper}  In Section 2 we review the required results from $p$-adic functional analysis, solid
mathematics, and Fontaine's theory of almost C-representations. In Section 3 we establish the topological properties of arithmetic pro-\'etale cohomology for Stein spaces. In Section 4 we recall the syntomic comparison theorem formulated on the Fargues-Fontaine curve, prove local Tate duality for syntomic sheaves, and deduce the arithmetic duality theorem.

\subsection*{Notations and conventions}

 We will follow the conventions from  \cite{curves} and work with condensed mathematics: The group cohomology appearing below is  the condensed group cohomology (cf. \cite[Sec. 3.1]{curves}). The condensed pro-\'etale cohomology of a rigid analytic variety is defined in \cite[Definition 4.1]{curves};  it is known to coincide with the condensed structures  associated with the classical pro-\'etale cohomology  (cf. \cite[Lemma 4.2]{curves}). We call a condensed set {\em classical} if it is the condensed set associated with a topological space.\par 
  Let $K$ be a finite extension of $\bQ_p$ with a ring of integers  $\cO_K$;  we denote by $G_K$ its absolute Galois group. Choose an algebraic closure $\overline{\bQ}_p$ of $\bQ_p$ and let $C:=\bC_p=\widehat{\overline{\bQ_p}}$, $\check{C}:=W(\overline{\bF}_p)[1/p]$. Note that there is an isomorphism of one dimensional $K$-vector spaces $\Hom_K(K,K)\cong\Hom_{\bQ_p}(K,\bQ_p)$ induced by the trace map $\mathrm{Tr}_{K/\bQ_p}$ and we use this identification without further mention. Partially proper rigid analytic spaces will be assumed to be separated and countable at infinity\footnote{It means the space $X$ can be written as a countable increasing union $X=\cup_n X_n$ such that each $X_n$ is a quasi-compact open subspace of $X$ and the inclusion $X_n\subseteq X_{n+1}$ factors through the adic compactification of $X_n$.}.
  
 \subsection*{Acknowledgment}
I would like to heartily thank Prof. Wieslawa Niziol, my Ph.D supervisor, for suggesting the problem to me, for many helpful conversations and for careful reading of a draft of the paper. The idea of the proof is inspired by her talk in Rapoport's 75th birthday conference. I thank Guido Bosco, Pierre Colmez, Gabriel Dospinescu, Sally Gilles, Long Liu, Lucas Mann, Xinyu Shao, Zhixiang Wu and Yicheng Zhou for helpful discussions. I also thank anonymous referees for many suggestions on improving the exposition.\par 
 The paper is written when the author is in his a Ph.D program in Sorbonne Universit\'e. This program has received funding from the European Union's Horizon 2020 research and innovation programme under the Marie Skłodowska-Curie grant agreement No 945332 and partially funded by the project “Group schemes, root systems, and related representations” founded by the European Union - NextGenerationEU through Romania’s National Recovery and Resilience Plan (PNRR) call no. PNRR-III-C9-2023- I8, Project CF159/31.07.2023.

\section{Preliminaries}

\subsection{Classical $p$-adic functional analysis}
 We recall some facts in $p$-adic functional analysis and refer to \cite[Sec. 2]{colmez2020cohomology}\cite[Sec. 1]{schneider2002locally} for details.\par 
Let $K$ be a finite extension of $\bQ_p$ with valuation ring $\cO_K$ and $C_K$ be the category of locally convex $K$-vector spaces, which is a quasi-abelian category \cite{prosmans2000derived}. We denote the category of left bounded complex by $C(C_K)$ and the associated derived $\infty$-category by $\sD(C_K)$. For an object $E\in C(C_K)$, the associated cohomology object is 
$$\widetilde{H}^n(E):=\tau_{\leq n}\tau_{\geq n}(E)=(\mathrm{coim}(d_{n-1})\to \ker(d_n)).$$
The cohomology group  $H^n(E)$ is taken in the category of $K$-vector spaces and, if necessary, equipped with sub-quotient topology. We call $\widetilde{H}^n(E)$ classical if the canonical map $\widetilde{H}^n(E)\to H^n(E)$ is an isomorphism.  \par 

  We denote the category of Hausdorff locally convex $K$-vector spaces by $C_K^H$. It is closed under direct sums and direct products. If $V\subseteq W$ are two locally convex $K$-vector spaces then $W/V\in C_K^H$ if and only if $V$ is closed in $W$. $C_K^H$ is quasi-abelian: cokernel and image are defined by closure of images in $C_K$.\\ \par 
  If $V,W\in C_K$, we denote by $L(V,W)$ the space of continuous linear maps from $V$ to $W$ and write $L_s(V,W), L_b(V,W)$ the space $L(V,W)$ equipped with weak and strong topology respectively. For $V\in C_K$, we write $V'_s=L_s(V,K)$ and $V^*:=V'_b=L_b(V,K)$. All of these dual spaces are Hausdorff (\cite[Proposition 9.1]{schneider2013nonarchimedean}). A space $V\in C_K^H$ is reflexive if the canonical map $V\to (V^*)^*$ is a topological isomorphism.\par

  Recall the definition of a spaces of compact type. 
  
  \begin{definition}[{\cite[Sec. 1, Definition]{schneider2002locally}}]\label{definitioncompacttype} Let $V,W\in C_K^H$.
  \begin{enumerate}
  	\item A subset $B\subseteq V$ is {\em compactoid} if for any open lattice $L\subset V$, there are finitely many vectors $v_1,...,v_m\in V$ such that $B\subseteq L+\cO_Kv_1+...+\cO_Kv_m$.
  	\item A continuous linear map $f: V\to W$ is {\em compact} if there is an open lattice $A$ in $V$ such that $f(A)$ is compactoid and the closure of $f(A)$ is complete.
  	\item V is said to be {\em of compact type} if it is the locally convex inductive limit of a sequence
  	$$V_1\stackrel{\iota_1}\to V_2\stackrel{\iota_2}\to V_3\to \cdots$$
  	of $V_n\in C_K^H$ ($n\in \bN$) with injective compact linear transition maps. We may assume $V_n$ to be Banach spaces (\cite[Proof of Proposition 16.10]{schneider2013nonarchimedean}).
  \end{enumerate}
  \end{definition} 

  There is a stereotype dual $V'_c$ for $V\in C_K$, defined as $L(V,K)$ equipped with the topology of compactoid convergence (see \cite[before Proposition 18.4 \& below Corollary 6.3]{schneider2013nonarchimedean}). We have continuous maps $V_b'\to V_c'\to V_s'$. If $V$ is a Banach (Smith) space, its stereotype dual is a Smith (Banach) space and $(V_c')_c'\cong V$. \\ \par

  We will freely use the following lemma.
  
  \begin{lemma}[{\cite[Remark 16.7]{schneider2013nonarchimedean}}]\label{lemmacomposition}  If $g:V\to W$ is a compact map and $f:V_1\to V$, $h:W\to W_1$ are two arbitrary continuous maps. Then $hgf:V_1\to W_1$ is compact.
  \end{lemma}

     By \cite[Theorem 1.1, Proposition 1.2]{schneider2002locally}, spaces of compact type are Hausdorff, complete, reflexive and the strong dual are Frechet spaces. If $V=\mathrm{colim}_n V_n$ as in the definition, then $V^*=\lim_n V_n^*$. A closed subspace and its relevant quotient of a space of compact type is still a space of compact type. The full subcategory of $C_K$ consisting of spaces of compact type is closed under countable direct sums.

  \begin{definition}[{\cite[Chap. 19, Definition]{schneider2013nonarchimedean}}] A space $V\in C_K$ is called {\em nuclear} if for any open lattice $L\subseteq V$ there exists another open lattice $M\subseteq L$ such that the map $\widehat{V_M}\to \widehat{V_L}$ is compact. Recall that $V_M$ is the vector subspace generated by $M$ equipped with the topology generated by the seminorm $p_M$ associated to $M$ and $\widehat{V_M}$ is its Hausdorff completion. 
  \end{definition}
  
  Nuclear Banach spaces are finite dimensional (\cite[Theorem 8.5.6]{perez2010locally}), nuclear Fr\'echet spaces are reflexive (\cite[Corollary 19.3]{schneider2013nonarchimedean}). A subspace of a nuclear space is nuclear and a quotient of a nuclear space by its closed subspace is nuclear. A projective limit and a countable inductive limit of nuclear spaces is nuclear (\cite[Theorem 8.5.7]{perez2010locally}).\\ \par 
  
  Recall that in the case of nuclear Fr\'echet spaces or spaces of compact type, stereotype dual and strong dual coincide:
  \begin{lemma}[{\cite[Lemma 2.8]{curves}}]\label{lemmasametopologyofdual}
  	If $X\in C_K$ is a nuclear Fr\'echet space or a space of compact type, then the map $X^*\to X_c'$ is a topological isomorphism.
  \end{lemma}

  We collect the following technical lemma for later use.

\begin{lemma} \label{lemmaextendcompact}
	If we have a morphism of strictly exact sequences of complete Hausdorff $K$-vector spaces
	\begin{equation}
		\begin{CD}
			0@>>>U@>>>V@>p>>W@>>>0\\
			@. @Vf_U VV @Vf_{V} VV @V f_W VV\\
			0@>>>U'@>>>V'@>p'>>W'@>>>0
		\end{CD}
	\end{equation}
	and assume that $U,U'$ or $W,W'$ are finite dimensional. If two of $f_U,f_V,f_W$ are compact then the remaining one is also compact.
\end{lemma}

\begin{proof}
	Note that in both cases the bottom exact sequence splits: when $W'$ is finite dimensional, one can choose any section of $p'$; when $U'$ is finite dimensional, this is \cite[Lemma 2.5]{curves}. First assume $U,U'$ are finite dimensional so $f_U$ is always compact. If $f_W$ is compact, the claim follows from \cite[Lemma 2.7]{curves}. If $f_V$ is compact, since $Im(U)$ is closed, $f_W$ is compact if and only if the composition map from $V$ to $W'$ is compact (\cite[Theorem 8.1.3(xi)]{perez2010locally}) which follows from $f_V$ is compact (\cite[Lemma 2.3]{curves}). \par 
		Assume $W,W'$ are finite dimensional so $f_W$ is always compact. If $f_V$ is compact, clearly $f_U$ is compact. If $f_U$ is compact, there is an open lattice $L_U$ whose image in $V'$ is compactoid. Choose a basis $\{w_1,...,w_n\}$ of $W$ and any section $s$ of $p$, then $L_U+\cO_Ks(w_1)+\cdots+\cO_Ks(w_n)$ is an open lattice of $V$ whose image is compactoid.  
	\end{proof}
	
	We recall the following topological Mittag-Leffler (TML) condition (\cite[Remark 13.2.4]{dieudonne1971elements}):
	
	\begin{definition}\label{deftopml}
		Let $\{V_n\}_{n\in \bN}$ be a countable projective system of Fr\'echet spaces (resp. spaces of compact type) with transition maps $f_{m,n}:V_m\to V_n$. We say that the system satisfies topological Mittag-Leffler (TML) condition if for any $n\in \bN$, there is $n_0\geq n$ such that for all $m\geq n_0$, we have 
		$$\overline{\im(f_{m,n}:V_m\to V_n})=\overline{\im(f_{n_0,n}:V_{n_0}\to V_n}).$$
		Equivalently, it means $\im(f_{m,n})\subseteq \im(f_{n_0,n})$ is dense.
	\end{definition}
	
	\begin{lemma}\label{lemtopml}
		Let $\{A_n,\alpha\}_{n\in \bN},\{B_n,\beta\}_{n\in \bN},\{C_n,\gamma\}_{n\in \bN}$ be projective systems of Fr\'echet spaces (resp. spaces of compact type). Assume there are compatible strict short exact sequences
		$$0\longrightarrow A_n\stackrel{p_n}\longrightarrow B_n\stackrel{q_n}\longrightarrow C_n\to 0$$
		\begin{enumerate}
			\item If $\{B_n\}_{n\in \bN}$ satisfies TML condition, then $\{C_n\}_{n\in \bN}$ also satisfies TML condition.
			\item If $\{A_n\}_{n\in \bN}$ and $\{C_n\}_{n\in \bN}$ satisfy TML condition, then $\{B_n\}_{n\in \bN}$ satisfies TML condition.
		\end{enumerate}
	\end{lemma} 
	
	\begin{remark}
		If $\{B_n\}_{n\in \bN}$, $\{C_n\}_{n\in \bN}$ satisfies TML condition, $\{A_n\}_{n\in \bN}$ does not necessarily satisfy TML condition.
	\end{remark}
	
	\begin{proof}
		(1) Fix $n\in \bN$, there is $n_0\geq n$ such that for all $m\geq n_0$, $\beta_{m,n}(B_m)$ is dense in $\beta_{n_0,n}(B_{n_0})$. Because $q_n$ are surjective and diagram commutes, we have
		$$q_n(\beta_{m,n}(B_m)=\gamma_{m,n}(C_m);\ q_n(\beta_{n_0,n}(B_m)=\gamma_{n_0,n}(C_{n_0}).$$
		Since $q_n$ is continuous, $\gamma_{m,n}(C_m)$ is dense in $\gamma_{n_0,n}(C_{n_0})$.
		
		(2) Fix $n$, we can find $a\geq n$ such that for all $r\geq a$, $\alpha_{r,n}(A_r)$ is dense in $\alpha_{a,n}(A_a)$. So we can find $b\geq a$ such that for all $s\geq b$, $\gamma_{s,a}(C_s)$ is dense in $\gamma_{b,a}(C_b)$. We will show that for any $m\geq b$, $\beta_{m,n}(B_m)$ is dense in $\beta_{b,n}(B_b)$. Let $L$ be an open lattice in $B_n$ and $x_b\in B_b$ is a point. Let $L_a:=\beta_{a,n}^{-1}(L)\subseteq B_a$. By open mapping theorem (for spaces of compact type see \cite{gilsdorf2020some}), $M_a:=q_a(L_a)$ is open in $C_a$. By definition there is $c_m\in C_m$ such that $\gamma_{b,a}(q_b(x))-\gamma_{m,a}(c_m)\in M_a$. Choose a lifting $y_m\in B_m$ of $c_m$ and $l_a\in L_a$ such that $\gamma_{b,a}(q_b(x))-\gamma_{m,a}(c_m)=q_a(l_a)$. Now we have
		$$v_a:=\beta_{b,a}(x_b)-\beta_{m,a}(y_m)-l_a\in \ker(q_a)=p_a(A_a)$$
		Let $v_a=p_a(z_a)$ for some $z_a\in A_a$. Apply TML condition of $A_n$ to the open $p_n^{-1}(L)$, we can find $a_m\in A_m$ such that $\alpha_{a,n}(z_a)-\alpha_{m,n}(a_m)\in p^{-1}_n(L)$. So 
		$$t=x-\beta_{m,n}(y_m+p_m(a_m))=\beta_{a,n}(l_a)+p_n(\alpha_{a,n}(z_a)-\alpha_{m,n}(a_m))\in L$$
		Since $L$ is arbitrary, we can conclude.
	\end{proof}

  \subsection{Solid functional analysis}

   We briefly review some solid functional analysis based on the condensed mathematics introduced by Clausen and Scholze. 
  
  \begin{definition}[{\cite[Definition 2.1, 2.11]{scholze2019condensed}}] 
  A condensed set/group/ring $\cdots$ is a sheaf over $*_{\proet}$, the pro-\'etale site of a geometric point, with values in $Sets$, $Groups$, $Rings \cdots$.
  \end{definition}
  
  \begin{remark}
    Since we will be interested in cohomology of Stein spaces and will deal with concrete objects, it is enough to cut off by an uncountable strong limit cardinal $\kappa$	 \cite[Lecture 2]{scholze2019condensed} (See \cite[Sec. 2.1]{mann2022p} and \cite[Remark 3.2.4]{camargo2024analytic} for discussion of (derived) internal Homs). We will not mention this $\kappa$ later and ignore all set theoritical issues.
  \end{remark}
  
  \begin{proposition}[{\cite[Example 1.5, Proposition 1.7]{scholze2019condensed},\cite[Proposition 1.2]{scholze2019lectures}}]\label{propositioncondensedbase}
  Let $Top$ be the category of topological spaces. There is a functor 
  $$(\_): Top\to CondSet,\ T\mapsto \underline{T}=\sC(-,T)$$
  where for a profinite set $S$, $\sC(S,T)$ is the set of continuous functions from $S$ to $T$.
  \begin{enumerate}
  	\item The functor has a left adjoint $X\to X(*)_{top}$ sending a condensed set $X$ to the  underlying set $X(*)$ equipped with the quotient topology arising from 
  	$$\coprod_{S,a\in X(S)} S\stackrel{a}\longrightarrow X(*).$$
  	\item The functor $(\_)$ is fully faithful when restricted to compactly generated topological spaces.
  	\item The functor induces an equivalence between the category of compact Hausdorff spaces and quasicompact quasiseparated condensed sets.
  	\item The functor induces a fully faithful functor from the category of compactly generated weak
			Hausdorff spaces to quasiseparated condensed sets. The category of quasiseparated condensed sets is equivalent to the category of ind-compact Hausdorff spaces 'colim' $X_i$  where all transition maps are closed immersions. If $\{X_n\}_{n\in \bN}$ is an ind-system of compact Hausdorff spaces with closed immersion transition maps and $X=\mathrm{colim} X_i$ as a topological space, then colim$\underline{X_i}\to \underline{X}$ is an isomorphism of condensed sets.  
  \end{enumerate}   
  \end{proposition}
  
  Now we pass to analytic rings and solid modules
  \begin{proposition}[{\cite[Proposition 7.5]{scholze2019condensed}}]\label{propositionsolid} Let $(A,\cM_A$) be a commutative analytic ring (see \cite[Definition 7.4]{scholze2019condensed})
  \begin{enumerate}
  	\item Let
  	$$\Mod_{(A,\cM_A)}^{\ssolid}\subseteq \Mod_A^{\scond}$$
  	be the full subcategory of all $A$-modules $M$ in condensed abelian groups such that for every extremally disconnected set $S$, the map 
  	$$\Hom_A(\cM_A[S],M)\to M(S)$$
  	is an isomorphism. It is an abelian category stable under all limits, colimits and extensions. The inclusion has a left adjoint
  	\begin{equation}\label{formulasolidify}
  	\Mod_A^{\scond}\to \Mod_{(A,\cM_A)}^{\ssolid}:\ M\mapsto M\otimes_A (A,\cM_A)
  	\end{equation}
  	which preserves all colimits. There exists a unique symmetric monoidal structure on $\Mod_{(A,\cM_A)}^{\ssolid}$ such that the left adjoint is symmetric monoidal.
  	\item We use $D(A,\cM_A)$ to denote $D(\Mod_{(A,\cM_A)}^{\ssolid})$. Then the functor
  	\begin{equation} \label{formulainclusion} D(A,\cM_A)\to D(\Mod_A^{\scond})\end{equation}
  	is fully faithful. The essential image is stable under all limits and colimits. It contains complexes $C\in D(\Mod_A^{\scond})$ such that for all extremally disconnected $S$, the map 
  	$$R\Hom_{A}(\cM_A[S],C)\to R\Hom_{A}(A[S],C)$$
  	is an isomorphism. An object $C$ is in $D(A,\cM_A)$ if and only if all $H^i(C)\in \Mod_{(A,\cM_A)}^{\ssolid}$. The functor \eqref{formulainclusion} has a left adjoint 
  	$$D(\Mod_A^{\scond})\to D(A,\cM_A):\ C\to C\otimes_A^L (A,\cM_A)$$
  	which is the left derived functor of $M\mapsto M\otimes_A(A,\cM_A)$.  
  \end{enumerate} 
  \end{proposition}

We recall the definition of solid nuclearity
\begin{definition}[{\cite[Definition 13.10]{scholze2019lectures}}]
    Let $(A,\cM)$ be an analytic ring. A complex $C\in D(A,\cM)$ is solid nuclear if for all extremally disconnected profinite set $S$, the natural map
    $$(A[S]^\vee\otimes_{(A,\cM)}^L C)(*)\to C(S)$$
    is an isomorphism in $D(\Ab)$ where $A[S]^\vee:=\underline{\Hom}_A(A[S],A)$. Let $D\nuc(A,\cM)$ be the full $\infty$-subcategory of $D(A,\cM)$ spanned by the nuclear complexes.
\end{definition}

\begin{remark}
    There is a definition of solid nuclearity at the abelian level (\cite[Definition A.40]{bosco2021p}). Assume $F$ is a non-archimedean local field and $A$ is a Banach (or more generally solid nuclear) $F$-algebra. Consider the analytic ring $(A,\cM)=(A,\bZ)_\square$. Then an object $M\in \Mod(A,\cM)$ is solid nuclear if and only if $M[0]$ is solid nuclear in $D(A,\cM)$. An object $C\in D(A,\cM)$ is solid nuclear if and only if $H^i(C)$ is solid nuclear for all $i$ (\cite[Theorem A.17]{bosco2023rational}). So we hope there will be no ambiguity in this case.
\end{remark}

\begin{notation}
 We call a solid $K$-vector space solid Banach (resp. solid Fr\'echet, nuclear, a solid space of compact type) if it is of the form $\underline{V}$ where $V$ is a classical Banach space (resp. Fr\'echet space, nuclear space, a space of compact type). Note that the definition of nuclear space is different from solid nuclearity. For example, infinite dimensional Banach $K$-vector spaces are all solidly nuclear but not nuclear.
\end{notation}

\begin{example}
\begin{enumerate}
	\item (\cite[Theorem 5.8]{scholze2019condensed}) The fundamental example of an analytic ring is $\bZ_\square=(\bZ,\bZ_\square)$ where for an extremally disconnected set $S=\lim_i S_i$, $\bZ_\square[S]=\lim_i \bZ[S_i]$.
	\item (\cite[Proposition 7.9]{scholze2019condensed}) Let $K$ be a finite extension of $\bQ_p$, there are analytic rings $(\cO_K,\cO_{K,\square})$ and $(K,K_{\square})$ such that for any extremally disconnected $S=\lim_i S_i$ 
	$$\cO_{K,\square}[S]=\lim_i \cO_K[S_i],\ K_{\square}[S]=K\otimes_{\cO_K} \cO_{K,\square}[S].$$
	In particular, it applies to the profinite set (group) $G_K$. Note that the analytic ring structure on $(\cO_K,\cO_{K,\square})$ is induced from $\bZ_\square$ (cf. \cite[Lemma A.19]{bosco2021p}). We will use $D(K_\square)$ to denote $D(K,K_\square)$.
\end{enumerate}
\end{example}

We will implicitly use the following results:

\begin{theorem}
    The functor $V\to \underline{V}$ induces an equivalence between classical and solid Banach (resp. Smith, resp. Fr\'echet, resp. compact type) $K$-vector spaces.
\end{theorem}

\begin{proof}
    This is \cite[Proposition 3.5, Lemma 3.24, Theorem 3.40]{rodrigues2022solid} and \cite[Corollary 1.4]{schneider2002locally}.
\end{proof}

\begin{lemma}[{\cite[Lemma 2.18, 2.20]{curves}}]\label{lemmausualtocondensed}
A sequence of solid Fr\'echet spaces (or spaces of compact type)	
$$0\to \underline{V_1}\to \underline{V_2}\to \underline{V_3}\to 0$$ 
is exact if and only if the sequence of topological vector spaces 
$$0\to V_1\to V_2\to V_3\to 0$$
is strictly exact.
\end{lemma}
\vspace{3pt}

The following properties of solid tensor product are often used.
  
  \begin{proposition}[{\cite[Corollary A. 65, 67]{bosco2021p}}]$ $\label{propositionlimitsolidtensorexchange}
  	\begin{enumerate}
  		\item  Any Fr\'echet $K$-vector space is flat for the solid tensor product $\otimes_K^\square$.
  		\item  Let $\{V_n\}$ be a countable inverse system of solid nuclear $K$-vector spaces and $W$ be a $K$-Fr\'echet space, then
  		$$(\lim_n V_n)\otimes^\square_K W\cong \lim_n(V_n\otimes^\square_K W).$$
  		\item  Let $\{\cV_n\}_n$ be a countable inverse system of objects in $D(K_\square)$ and representable by a complex of solid nuclear $K$-vector spaces. Assume $\cW\in D(K_\square)$ is represented by a bounded above complex of Fr\'echet space. Then
  		$$ (R\lim_n \cV_n)\otimes^{L_\square}_K \cW\simeq  R\lim_n(\cV_n\otimes^{L_\square}_K \cW).$$
  	\end{enumerate}
  \end{proposition}

  \subsection{Almost $C$-representations}
  
    In \cite{katopresque}, Fontaine introduced $\sC(G_K)$, the abelian category of almost $C$-representations, which is a strict subcategory of the exact category $\cB(G_K)$ of Banach representations of $G_K$ (cf. \cite[Theorem 5.1]{katopresque}). We briefly recall the definition and two important results we need.\\
    
    Recall that two Banach representations $W_1$ and $W_2$ are almost isomorphic if there
exist two finite dimensional $\bQ_p$-vector spaces $V_i \subseteq W_i$, $i= 1,2$, stable under the $G_K$-action such that $W_1/V_1\cong 
W_2/V_2$. An almost $C$-representation is a Banach representation which is almost isomorphic to $C^d$
for some $d\in \bN$. Denote by $\sC(G_K)$ the category of almost $C$-representations. 
  
  \begin{theorem}[{\cite[Theorem 6.1, Proposition 6.7, Proposition 6.11]{katopresque}}]\label{theoremalmostcpdual} Suppose $X,Y \in \sC(G_K)$.
  	\begin{enumerate}
  		\item For each $n\in \bN$, the $\bQ_p$-vector space $\Ext^n_{\sC(G_K)}(X,Y)$\footnote{The extension groups considered there are Yoneda extensions which coincide with the Hom in the derived category, even if the category does not contain enough injectives (See \cite[\href{https://stacks.math.columbia.edu/tag/06XP}{Tag 06XP}]{stacks-project}).} is finite dimensional and vanishes when $n\geq 3$. One has that 
        \begin{equation}\label{eulercharacter}
        \sum_{i=0}^2 (-1)^i \dim_{\bQ_p} \Ext^i_{\sC(G_K)}(X,Y)=-[K:\bQ_p]h(X)h(Y)
         \end{equation}
         where $h(X),h(Y)\in \bZ$ is the height of the representation.
  		\item If $X$ is a $p$-adic representation of $G_K$, then there is a natural isomorphism
  		$$ \Ext^n_{\sC(G_K)}(X,Y)\cong H^n_{\cont}(G_K,X^*\otimes Y)$$  	
  		\item There is a natural trace map $\Ext^2_{\sC(G_K)}(X,X(1))\to \bQ_p$ such that the map 
  	$$\Ext^n_{\sC(G_K)}(X,Y)\times \Ext^{2-n}_{\sC(G_K)}(Y,X(1))\to \Ext^2_{\sC(G_K)}(X,X(1))\to \bQ_p$$
  	is a perfect pairing.
  		\end{enumerate}
  \end{theorem}
  
  Let $\cM(G_K)$ be the category of $G_K$-equivariant coherent $\cO_{\FFalg}$-modules where $\FFalg$ is the algebraic Fargues-Fontaine curve. Let $\cM^{\geq 0}(G_K)\subseteq \cM(G_K)$ and $\sC^{\geq 0}(G_K)\subseteq \sC(G_K)$ be the subcategories of effective objects (See \cite[Page 2]{fontaine2019almost}). Recall that the objects of $\cM^{\geq 0}(G_K)$ are coherent $\cO_{\FF^{\alg}}[G_K]$-modules  with HN-slope $\geq 0$ and the objects of $\sC^{\geq 0}(G_K)$ are subobjects (in $\sC(G_K)$) of $B_{\dR}^+$-representations of finite length. The following theorem is proved by Fontaine and will often be used in this paper:
  
  \begin{theorem}[{\cite[Theorem 5.9, 6.4]{fontaine2019almost}}] \label{theoremfontaineequivalence}
  There is an equivalence of categories 
  $$\begin{CD}
\cM^{\geq 0}(G_K)@>H^0(\FFalg,-)>\cong >\sC^{\geq 0}(G_K)
\end{CD}$$
which induces an equivalence of triangulated categories
$$ 
  	R\Gamma(\FFalg,-): D^b(\cM(G_K))\stackrel{\simeq}\longrightarrow D^b(\sC(G_K)).
$$
  \end{theorem}
  
  \begin{proof}
  Only the statement that the equivalence of triangulated categories is induced by $R\Gamma(\FFalg,-)$ is not stated in the cited theorem, but it is implicit in the proof: Fontaine proved in  \cite[Theorem 6.4]{fontaine2019almost} that the exact embedding $\cM^{\geq 0}(G_K) \to \cM(G_K)$ is left big (cf. \cite[Sec. 6.2 Page 719]{fontaine2019almost}) with respect to $\cM^{\infty}(G_K)$ (objects of $\cM(G_K)$ whose underlying $\cO_{\FFalg}$-module is torsion). So for each $A\in \cM(G_K)$ there exists $B\in \cM^{\geq 0}(G_K)$ and $C\in \cM^{\infty}(G_K)$ together with a short exact sequence 
  $$0\to A\to B\to C\to 0.$$
  Since $B,C\in \cM^{\geq 0}(G_K)$, the image of $A$ in $D^b(\sC(G_K))$ is $[H^0(\FFalg,B)\to H^0(\FFalg,C)]$. While $H^0(\FFalg,B)[0]\simeq R\Gamma(\FFalg,B)$ and $H^0(\FFalg,C)[0]\simeq R\Gamma(\FFalg,C)$, so the complex is equivalent to $R\Gamma(\FFalg,A)$.
  \end{proof}

\section{Topology on Cohomology Groups}
 Let $X$ be a smooth, geometrically connected Stein space of dimension $d$ over $K$. 
 In this section, we study the topological property of (compactly supported) pro-\'etale cohomologies of smooth Stein spaces.

 To begin with, we recall the definition of compactly supported pro-\'etale cohomology in \cite[Sec. 2.1]{achinger2025compactlysupportedpadicproetale}.
 For a smooth partially proper rigid analytic variety $X$ over $K$, the cohomology of the boundary is defined as 
$$R\Gamma_{\proet}(\partial X,\bQ_p(j)):=\colim_{U\subseteq X} R\Gamma_{\proet}(X\backslash U,\bQ_p(j))$$
where $U\subseteq X$ goes through admissible quasi-compact opens of $X$. The compactly supported pro-\'etale cohomology is defined as the fiber
\begin{equation}\label{formuladefofcompactlysupport}
R\Gamma_{\proet,c}(X,\bQ_p(j)):=[R\Gamma_{\proet}(X,\bQ_p(j))\to R\Gamma_{\proet}(\partial X,\bQ_p(j))].
\end{equation}

Suppose $U\subseteq X$ is a strict open immersion\footnote{It means the open immersion factors through the adic universal compactification of $U$}. There is a morphism  (cf. \cite[Sec. 2.1]{achinger2025compactlysupportedpadicproetale})
$$R\Gamma_{\proet,c}(U,\bQ_p(j))\to R\Gamma_{\proet,c}(X,\bQ_p(j)).$$

In \cite{colmez2020cohomology}, \cite{niziol2021cohomology}, the authors defined the (complete) Hyodo-Kato cohomology $R\Gamma_{\HK}(X_C)$. In \cite[Sec. 3.2]{achinger2025compactlysupportedpadicproetale}, the authors defined $R\Gamma_{\HK,c}(X_C)$, the (complete) compactly supported Hyodo-Kato cohomology of $X_C$.  It is defined  similarly to \eqref{formuladefofcompactlysupport}. 
 
 The proof of the topological properties relies on the following comparison theorems.

 \begin{theorem}[{\cite[Theorem 6.9]{niziol2021cohomology}, \cite[Theorem 7.7]{bosco2023rational}, \cite[Theorem 6.13]{achinger2025compactlysupportedpadicproetale}}]\label{theoremclassicalperiod}\
 Let $r\geq 0$,
 \begin{enumerate}
 	\item There is a $G_K$-equivariant exact sequence of solid $\bQ_p$-vector spaces
 	\begin{equation}\label{exactsequence1}
 		0\to \Omega^{r-1}(X_C)/ker(d)\to H^r_{\proet}(X_C,\bQ_p(r))\to (H^r_{\HK}(X_C)\otimes^\square_{\check{C}} \hat{\mathbf{B}}_{\st}^+)^{\varphi=p^r, N=0}\to 0
 	\end{equation}
 	where 
 	$$\hat{\mathbf{B}}^+_{\st}:=A_{\crys}\langle t_p[\tilde{p}]^{-1}-1\rangle^{\wedge}[1/p]$$
  is a completed version of the period ring $\mathbf{B}_{\st}^+$ (See \cite[Sec 2.1.1]{niziol2021cohomology} for notations and a more geometric definition). Note that $\hat{\mathbf{B}}_{\st}^+$ is a Banach $\check{C}$-algebra and morphisms from $\mathbf{B}_{\st}^+$ and $\hat{\mathbf{B}}_{\st}^+$ to $\mathbf{B}_{\dR}^+$ are all normalized at $p$ (See $loc.\ cit$).\\
 	
 	\item Put 
 	$$DR_c(X_C,r):= R\Gamma_{\dR,c}(X_C/\mathbf{B}^+_{\dR})/F^r$$ $$HK_c(X_C,r):=[R\Gamma_{\HK,c}(X_C)\otimes_{\check{C}}^{L_\square} \hat{\mathbf{B}}_{\st}^+]^{\varphi=p^r, N=0}.$$ 
 	Then there is a $G_K$-equivariant exact sequence of solid $\bQ_p$-vector spaces
 	\begin{equation}\label{exactsequence2}
 	\begin{aligned}
 		&H^{r-1}HK_c(X_C,r)\to H^{r-1}DR_c(X_C,r)\to H^r_{\proet,c}(X_C,\bQ_p(r))\\
 		&\to H^rHK_c(X_C,r)\to H^rDR_c(X_C,r).
 		\end{aligned}
 	\end{equation} 
 \end{enumerate} 	
 \end{theorem} 

 \begin{remark}\label{remarkonderhamcomplex}
     For each $r\in \bN$, recall that because of acyclicity of coherent cohomology on Stein spaces, we have the following presentation (for example see \cite[Sec. 3.1.2]{geometricdual}) for the truncated $\mathbf{B}_{\dR}^+$-complex when $X$ is a smooth Stein space (if $a<0$, we let $t^a\mathbf{B}_{\dR}^+:=F^a\mathbf{B}_{\dR}^+=\mathbf{B}_{\dR}^+$) :
 \begin{equation}\label{formulaofDRcomplex}
 \begin{aligned}
 &DR(X_C,r):=R\Gamma(X_C/\mathbf{B}_{\dR}^+)/F^r\\ 
 &\simeq \cO(X)\otimes_K^\square (\mathbf{B}_{\dR}^+/t^r)\to \Omega^1(X)\otimes_K^\square (\mathbf{B}_{\dR}^+/t^{r-1})\to \cdots \to \Omega^d(X)\otimes_K^\square (\mathbf{B}_{\dR}^+/t^{r-d}) \\
 &DR_c(X_C,r):= R\Gamma_{c}(X_C/B^+_{\dR})/F^r\\
 &\simeq [H^d_c(X,\cO)\otimes_K^\square (\mathbf{B}_{\dR}^+/t^r)\to H^d_c(X,\Omega^1)\otimes^\square_K (\mathbf{B}_{\dR}^+/t^{r-1})\to\cdots H^d_c(X,\Omega^d)\otimes_K^\square (\mathbf{B}_{\dR}^+/t^{r-d})][-d]
 \end{aligned}\end{equation}
 \end{remark}
 
 We will use the following lemmas on compatibility between continuous group cohomology and condensed group cohomology. In this paper, except otherwise specified, we will use $H^i(G,V)$ and $R\Gamma(G,V)$ to denote condensed group cohomology.
 
 \begin{lemma}[{\cite[Proposition B.2]{bosco2021p}}]
 	Let $G$ be a profinite group and let $V$ be a $G$-module in solid abelian groups. Then
 	\begin{enumerate}
 		\item The complex $R\Gamma(G, V)$ is quasi-isomorphic to the complex of solid abelian groups
 		$$V\to \underline{\Hom}(\bZ[G],V)\to \underline{\Hom}(\bZ[G\times G],V)\to \cdots $$
 		\item If $V=V_{top}$ where $V_{top}$ is a T1 topological $G$-module over $\bZ$, then for each $i\geq 0$ we have a natural isomorphism of abelian groups 
 		$$H^i(G,V)(*)\cong H^i_{\cont}(G,V). $$
 	\end{enumerate}
 	 \end{lemma}

 \begin{lemma}\label{lemmacondenseandcontinuous}
 	Let $G$ be a profinite group and $V$ be a Banach $\bQ_p$-vector space with a continuous action of $G$. Then
 	  	$$ \underline{\tilde{H}^i(G,V)}\cong  H^i(G,\underline{V}).$$
 	  	Assume further that $ H^i_{\cont}(G,V)$ is finite dimensional, then   	  	$$\underline{\tilde{H}^i(G,V)}\cong  \underline{H^i_{\cont}(G,V)}\cong H^i(G,\underline{V}). $$
 \end{lemma}
 
 \begin{proof}
 	Let $S$ be a profinite set. The condensed group cohomology is computed by
 	$$R\Gamma(G,\underline{V}): n\to \underline{\Hom}(\bZ[G^n],\underline{V}).$$
 	$$ \underline{\Hom}(\bZ[G^n],\underline{V})(S)=\Hom(\bZ[G^n\times S],\underline{V})\cong \sC(G^n\times S, V)$$
 	where $\sC(-,-)$ denotes the space of continuous maps with compact open topology. 
 	The continuous group cohomology is computed by the continuous cochain complex
 	$$\underline{R\Gamma_{\cont}(G,V)}: n\to \underline{\sC(G^n,V)}$$
 	$$\underline{\sC(G^n,V)}(S)=\sC(S,\sC(G^n,V))\cong \sC(G^n\times S,V)$$
	where we used compact open topology is exponential. The compact open topology on $\sC(G^n,V)$ coincides with the supremum topology under which it is a Banach space. The first claim now follows from Lemma \ref{lemmausualtocondensed}.
	
	  For the second claim, note that if $f: A\to B$ is a continuous map of Banach $\bQ_p$-vector spaces such that $M:=\coker(f)$ is finite dimensional, then $f$ is strict. This is because if we let $\pi:B\to M$ be the projection, then $\ker(\pi)$ is closed in $B$ as $M$ is Hausdorff. So $A\to \ker(\pi)$ is a continuous surjective map of Banach spaces, the open mapping theorem implies the strictness of $f$. The second claim follows if we take $A=\sC(G^{i-1},V)$, $B=\ker(d_i)$ and $f=d_{i-1}$
 \end{proof}
 

 \subsection{Arithmetic pro-\'etale cohomology}
Let $X$ be a smooth, geometrically connected Stein space of dimension $d$ over $K$. We start by recalling a well-known result.
 
\begin{lemma} \label{lemmafundanuclearfrechet}
	The spaces $\Omega^i(X)$ are nuclear Fr\'echet spaces. 
	\end{lemma} 

\begin{proof}
	
    We first deal with the Tate algebra: Assume $\delta\in p^{\bQ}$ is greater than $1$, we show the map
	$$j_\delta: T_n(\delta):=K\langle \delta^{-1} X_1,...,\delta^{-1} X_n\rangle \ \longrightarrow K\langle T_1,...,T_n\rangle=: T_n$$
	is a compact morphism. Denote their Gauss norms by $|\cdot|_{\delta}$ and $|\cdot|_1$ respectively. Let $L'\subseteq T_n$ be any open lattice and $L=\{f\in T_n(\delta): |f|_{\delta}\leq 1\}$, then there is some $s$ such that $B_s:=\{f\in T_n: |f|_1\leq \delta^{-s}\}\subseteq L'$. For $f=\sum_I a_IX^I\in L$ with $a_I=0$ when $|I|\leq s$, we have $|f|_1\leq \delta^{-s}|f|_\delta\leq \delta^{-s}$, i.e $f\in B_s$. So $j_\delta(L)\subseteq L'+\sum_{|I|\leq s} \cO_K x^I$ as desired.
	
	For general $X$, we choose a Stein cover $\{X_n\}_n$ of $X$ by rigid affinoids. Then each $X_n\subseteq X_{n+1}$ can be represented by $T_m(\delta)/I\to T_m/IT_m$ for some $m$ and $\delta>1$. Since $\Omega^i(X_n)$ (resp. $\Omega^i(X_{n+1})$) is a finitely presented module of $T_m$ (resp. $T_m(\delta))$ (affinoid algebras are Noetherian), the result follows from \cite[Theorem 8.1.3(xi)]{perez2010locally} and the computation above.
		\end{proof}

  Recall the following terminology in the theory of rigid analytic varieties: a naive interior of a smooth (dagger) affinoid variety is a Stein subvariety whose complement is open and quasi-compact. Although a naive interior itself is not quasi-compact, it has finite dimensional de Rham cohomology (\cite[Theorem A]{finitederham}). It is easy to see that, for a pair of dagger affinoids $X_1 \Subset X_2$ there exists a naive interior $X_2^0 \subset X_2$ such that $X_1 \subset X_2^0 \subset X_2$. We will say that $X_2^0$ is adapted to $X_1$.
  \bigskip
  
\begin{theorem} \label{theorem1}  
Let $X$ be a smooth geometrically connected Stein space of dimension $d$ over $K$. Let $\{X_n\}$ be a strictly increasing open cover of $X$ by adapted naive interiors of dagger affinoids. 
    \begin{enumerate}
        \item  For $i,j\in \bZ$, $H^i_{\proet}(X_n,\bQ_p(j))$ is a nuclear Fr\'echet space and transition maps
         $$H^i_{\proet}(X_{n+1},\bQ_p(j))\to H^i_{\proet}(X_n,\bQ_p(j))$$
        are compact morphisms.
        
		\item For $i,j\in \bZ$, the system $\{H^i_{\proet}(X_n,\bQ_p(j)) \}_{n\in \bN}$ satisfies topological Mittag-Leffler condition. In particular, we have	$R^1\lim_n H^i_{\proet}(X_n,\bQ_p(j))=0$, 
		 $$H^i_{\proet}(X,\bQ_p(j))\cong\lim_n H^i_{\proet}(X_n,\bQ_p(j))$$
		 and the cohomology groups $H^i_{\proet}(X,\bQ_p(j))$ are nuclear Fr\'echet.
        \end{enumerate}
\end{theorem}

\begin{proof}
	We have the Hochschild-Serre spectral sequence \cite[(4.5)]{curves} 
	\begin{equation}\label{spectralsequenceHS}
    E_{2,n}^{a,b}=H^a(G_K,H^b_{\mathrm{proet}}(X_{n,C},\bQ_p(j)))\Rightarrow H^{a+b}_{\mathrm{proet}}(X_n,\bQ_p(j)).
    \end{equation}
	and that $E_{2,n}^{a,b}=0$ when $a>2$.  The spectral sequence degenerates on the $E_3$-page and we have 
	$$E_{3,n}^{a,b}=
	\begin{cases}
		\ker(d_{0,b}), & a=0\\
		E_{2,n}^{1,b}, & a=1 \\
		\coker(d_{0,b+1}), & a=2
		\end{cases}
		$$
		where $d_{a,b}$ is the differential map from $E_{2,n}^{a,b}$ to $E_{2,n}^{a+2,b-1}$.

	 Set $s=j-b$ and write the $s$-th Tate twist $(-)\otimes \bQ_p(s)$ as $(-)(s)$. Apply Galois cohomology to \eqref{exactsequence1} and denote $HK^b(X_C,r):=(H^b_{\HK}(X_C)\otimes^\square_{\check{C}} \hat{\mathbf{B}}_{\st}^+)^{\varphi=p^r, N=0}$. We get a long exact sequence  
	$$ \begin{aligned}
		\to & H^{a-1}(G_K,HK^b_{}(X_{n,C},b)(s))\stackrel{\partial_{a-1}}\longrightarrow H^a(G_K,(\Omega^{b-1}(X_{n,C})/\ker{d})(s))\to \\
		& E_{2,n}^{a,b}\to H^a(G_K,HK^b(X_{n,C},b)(s))\stackrel{\partial_a}\longrightarrow H^{a+1}(G_K,(\Omega^{b-1}(X_{n,C})/\ker d)(s))\to.
	\end{aligned}
	$$ 
	So we have a short exact sequence
	\begin{equation}\label{E2}
		0\to \coker(\partial_{a-1})\to E_{2,n}^{a,b}\to \ker(\partial_a)\to 0.
	\end{equation}
	Since $HK^b_{}(X_{n,C},b)(s)$ is a Banach-Colmez space (Hyodo-Kato cohomology has finite rank by Hyodo-Kato isomorphism with de Rham cohomology),  we claim that the groups $H^a(G_K,HK^b(X_{n,C},b)(s))$ are finite dimensional over $\bQ_p$, so is $\ker(\partial_a)$. This follows from Theorem \ref{theoremalmostcpdual} and \ref{theoremfontaineequivalence}: the second theorem tells the Galois representation on this Banach-Colmez space is an almost $C$-representation and the first theorem gives that Galois cohomology is finite dimensional. 
	Using generalized Tate's isomorphism (\cite[Proposition 3.15]{curves})  and $\Omega^{b-1}(X_{n,C})/\ker(d)\cong \Omega^{b-1}(X_n)/\ker(d)\otimes_K^\square C$, we want to show
	\begin{equation} H^a(G_K,\Omega^{b-1}(X_{n,C})/\ker(d)(s))=\begin{cases}
		\Omega^{b-1}(X_{n})/\ker(d)  & a=0,1;b=j\\
		0 & others.
	\end{cases}\label{differential}\end{equation}
	are nuclear Fr\'echet spaces. This follows from Lemma \ref{lemmafundanuclearfrechet} and the fact that $ker(d)$ is a closed subspace.\par 
		
	Therefore, we know $E_{2,n}^{a,b}, E_{3,n}^{a,b}$ are finite dimensional $\bQ_p$-vector spaces except when $a=0,1;b=j$. In fact, $E_{3,n}^{a,b}$ are nuclear Fr\'echet spaces. We can assume $a=0,1;b=j$ but in this case $E_{2,n}^{a,b}$ is an extension of a finite dimensional $\bQ_p$-vector space by a nuclear Fr\'echet space, which is nuclear Fr\'echet (it splits) and $\ker(d_{0,j})$ is a closed subspace of $E^{0,j}_{2,n}$.
	
	  We show that $H^i_{\proet}(X_n,\bQ_p(j))$ are nuclear Fr\'echet spaces. The spectral sequence \eqref{spectralsequenceHS} defines a decreasing filtration $F^p$ on $H^i_{\proet}(X_n,\bQ_p(j))$ whose graded pieces are $E_{3,n}^{0,i}, E_{3,n}^{1,i-1}$ and $E_{3,n}^{2,i-2}$. By the computation above, only one of these three spaces is possibly not finite dimensional over $\bQ_p$ and in all cases they are all nuclear Fr\'echet. Hence $H_{\proet}^i(X_n,\bQ_p(j))$ is finite dimensional or can be written as an extension of a nuclear Fr\'echet space by a finite dimensional space ($E^{2,i-2}_{3,n}$ is always finite dimensional), which is in fact nuclear Fr\'echet. Only the second case needs to be clarified: Assume we have an extension 
          \begin{equation}\label{extfrechet}
          0\to V\to H_{\proet}^i(X_n,\bQ_p(j))\stackrel{q_1}\to F_1\to 0
          \end{equation}
          where $V$ is a finite dimensional $\bQ_p$-vector space and $F_1$ is a solid nuclear Fr\'echet space. It suffices to show $H_{\proet}^i(X_n,\bQ_p(j))$ is a Fr\'echet space because then \eqref{extfrechet} splits by Lemma \ref{lemmausualtocondensed} and \cite[Lemma 2.5]{curves}.
          Since the classical pro-\'etale cohomology $R\Gamma_{\proet}(X_{n},\bQ_p(j))^{cl}$ can be represented by a complex of Fr\'echet spaces, there exists an exact sequence
          $$ F_3\to F_2\stackrel{q_2}\to H_{\proet}^i(X_n,\bQ_p(j))\to 0$$
          where $F_2, F_3$ are solid Fr\'echet spaces. Let $q_3=q_1\circ q_2$, then $q_3$ is induced by a continuous surjective map of Fr\'echet spaces. By the open mapping theorem, $\ker(q_3)$ is a closed subspace of $F_2$. From the exact sequence $0\to \ker(q_2)\to \ker(q_3)\to V\to 0$, we see $\ker(q_2)$ is a closed subspace of $\ker(q_3)$, thus is closed in $F_2$. It follows that $H_{\proet}^i(X_n,\bQ_p(j))$ is a Fr\'echet space.
          
          We show the transition maps are compact. By applying Lemma \ref{lemmaextendcompact}, this is reduced to show $E_{2,n+1}^{a,b}\to E_{2,n}^{a,b}$ is compact when $a=0,1$ and $b=j$. Use Lemma \ref{lemmaextendcompact} again, it is reduced to show transition maps $\Omega^{b-1}(X_{n+1})/\ker(d)\to \Omega^{b-1}(X_{n})/\ker(d)$ are compact. Then it is further reduced to show $\Omega^{b-1}(X_{n+1})\to \Omega^{b-1}(X_n)$ is compact. We can choose a rigid affinoid $X_n^0$ such that $X_n\subset X_n^0\Subset X_{n+1}$. The morphism factors as $\Omega^{b-1}(X_{n+1})\stackrel{f}\to \Omega^{b-1}(X^0_n)\to \Omega^{b-1}(X_n)$. Now $f$ is compact because $\Omega^{b-1}(X_{n+1})$ is nuclear and $\Omega^{b-1}(X^0_n)$ is a Banach space (\cite[Proposition 19.5]{schneider2013nonarchimedean}).
          
          Now we show that $\{H^i_{\proet}(X_n,\bQ_p(j))\}_{n\in \bN}$ satisfies topological Mittag-Leffler condition. By the analysis of the spectral sequence above, we already know $H^i_{\proet}(X_n,\bQ_p(j))$ is finite dimensional except for $i=j$ or $j+1$. So we will consider these two cases. 
          
          When $i=j+1$, by lemma \ref{lemtopml}(2), we need to show $\{E_{3,n}^{a,b}\}_{n\in \bN}$ satisfies topological Mittag-Leffler condition for $a+b=j+1$. In this case, $E^{0,j+1}_{3,n}$ and $E^{2,j-1}_{3,n}$ are finite dimensional and we have $E_{3,n}^{1,j}=E^{1,j}_{2,n}$. By \eqref{E2}, it suffices to show the claim for the system $\coker(\partial_{a-1})$ and by Lemma \ref{lemtopml}(1) it suffices to show it for the system $\{\Omega^{j-1}(X_n)/\ker(d)\}_n$. This follows from that the system $\{\Omega^{j-1}(X_n)\}_n$ satisfies topological Mittag-Leffler condition.
          
          When $i=j$, we need to show $\{H^i_{\proet}(X_n,\bQ_p(i))\}_{n\in \bN}$ satisfies topological Mittag-Leffler condition. For this, we use arithmetic syntomic comparison theorem (c.f \cite[Corollary 7.3]{colmez2020p}) that there is a strict exact sequence (See \cite[Theorem 7.7(2)]{colmez2020p}, note that for Stein $X$ all cohomology groups appear below are classical)
          $$0\to \Omega^{i-1}(X_n)/\im(d)\to H^i_{\proet}(X_n,\bQ_p(i))\to H^i(R\Gamma_{\HK}(X_n)^{\varphi=p^i,N=0})\to 0.$$
          The term  $H^i(R\Gamma_{\HK}(X_n)^{\varphi=p^i,N=0})$ is in fact finite dimensional over $\bQ_p$. Using the spectral sequence for $(\varphi,N)$-cohomology, it suffices to show $H^i_{\HK}(X_n)$ is finite dimensional over $\bQ_p$. By \cite[Proposition 4.26]{colmez2020p}, $H^i_{\HK}(X_n)=H^i_{\HK}(X_C)^{G_K}$ where $H^i_{\HK}(X_C)$ is a finite dimensional $K^{ur}$-vector space (the $G_K$-action is smooth).  and the result follows from Galois descent. So it suffices to show $\{\Omega^{i-1}(X_n)/\im(d)\}_n$ satisfies topological Mittag-Leffler condition. By \cite[Corollary 3.2]{grosse2004rham}, $\im(d)$ is closed in $\Omega^{i-1}(X_n)$, thus the result also follows from that the system $\{\Omega^{i-1}(X_n)\}_n$ satisfies topological Mittag-Leffler condition.

	  Since we know $\{H^i_{\proet}(X_n,\bQ_p(j))\}_{n\in \bN}$ is a projective system of nuclear Fr\'echet spaces satisfying topological Mittag-Leffler condition. We get vanishing of $R^1\lim$ and $H^i_{\proet}(X,\bQ_p(j))\cong\lim_n H^i_{\proet}(X_n,\bQ_p(j))$. The last claim follows from that a (countable) projective limit of nuclear Fr\'echet space is nuclear Fr\'echet \cite[Corollary 3.5.7, Theorem 8.5.7(iii)]{perez2010locally}.          
\end{proof}

\subsection{Arithmetic compactly supported pro-\'etale cohomology}

The second step is to deal with the compactly supported cohomology. Spaces of compact type have the property that when we take dual of spaces of compact type, no higher extension groups appear:

\begin{lemma}\label{lemmacompacttypedual}
  	Assume $V$ is a $K$-vector space of compact type, then $R\Hom_K(V,K)\simeq V^*$.  
  \end{lemma}
  
  \begin{proof}
  	In the terminology of \cite[Section 3.5]{rodrigues2022solid}, $V$ is an $LB$ space of compact type hence also an $LS$ space of compact type (\cite[Corollary 3.38]{rodrigues2022solid}). We write $V=\mathrm{colim}_n S_n$ as inductive limit of Smith spaces with injective trace class transition maps. So 
  	$$\begin{aligned}
  	& R\Hom_K(V,K)\simeq R\Hom_K(\mathrm{colim}_n(S_n),K)\simeq R\lim_n R\Hom_K(S_n,K)\\
  	 & \simeq R\lim_n \Hom_K(S_n,K)\simeq \lim_n \Hom_K(S_n,K)
  	\end{aligned}$$
  	where the third quasi-isomorphism follows as $S_n$ are projective objects. For the forth one, we apply topological Mittag-Leffler (\cite[Lemma A.37]{bosco2021p}). By duality between Banach and Smith spaces (\cite[Lemma 3.10]{rodrigues2022solid}), it is enough to see transition maps have dense images, which is \cite[Lemma 3.19]{rodrigues2022solid}. Now the claim follows from the equivalence between classical and condensed Banach (Smith) spaces (\cite[Proposition 3.5]{rodrigues2022solid}). 
  \end{proof}

We prove the following partial result on the topological properties of compactly supported cohomology and postpone the proof that $H^i_{\proet,c}(X,\bQ_p(j))$ is of compact type to Theorem \ref{theoremsecondmain}:

\begin{theorem} \label{theorem2}
	Let $X$ be a smooth geometrically connected Stein space of dimension $d$ over $K$. Let $\{X_n\}_{n\in \bN}$ be a strictly increasing open covering of X by adapted naive interiors of dagger affinoids. Then the cohomology groups $H^i_{\proet,c}(X_n,\bQ_p(j))$ are of compact type and $H^i_{\proet,c}(X,\bQ_p(j))$ is an inductive limit of spaces of compact type with compact transition maps.
\end{theorem}

\begin{proof} 

We will show that $H^i_{\proet,c}(X_n,\bQ_p(j))$ are spaces of compact type and the transition maps between them are compact morphisms.

	We have the Hochschild-Serre spectral sequence of compactly supported pro-\'etale cohomology (cf. \cite[Lemma 5.3]{curves})
	\begin{equation}
		E_{2,n}^{a,b}=H^a(G_K,H^b_{\proet,c}(X_{n,C},\bQ_p(j)))\Rightarrow H^{a+b}_{\proet,c}(X_n,\bQ_p(j))
	\end{equation}
	and $E^{a,b}_{2,n}=0$ when $a>2$. This spectral sequence degenerates on page $E_3$ and we have 
	$$E_{3,n}^{a,b}=
	\begin{cases}
		\ker(d_{0,b}), & a=0\\
		E_{2,n}^{1,b}, & a=1 \\
		\coker(d_{0,b+1}), & a=2.
		\end{cases}
		$$
		Consider the Galois equivariant exact sequence \eqref{exactsequence2} and twist it by $s:=j-b$, we get:
		\begin{equation*}
		\begin{aligned}
		\begin{CD}
			HK^{b-1}_{c}(X_{n,C},b)(s)&@>\phi_{b-1}(s)>>& H^{b-1}DR_{c}(X_{n,C},b)(s)&\longrightarrow H^b_{\proet,c}(X_{n,C},\bQ_p(j))\longrightarrow \\ HK^b_{c}(X_{n,C},b)(s)&@>\phi_b(s)>>& H^bDR_{c}(X_{n,C},b)(s)&
			\end{CD}
			\end{aligned}
		\end{equation*}
	Hence there is an exact sequence
	\begin{equation*}
		0\to \coker(\phi_{b-1})(s)\longrightarrow H^b_{\proet,c}(X_{n,C},\bQ_p(j))\longrightarrow \ker(\phi_b)(s)\to 0.
	\end{equation*}
	Taking Galois cohomology one gets a long exact sequence
	\begin{equation}
	\begin{aligned}
		\to H^{a-1}(G_K,\ker(\phi_{b})(s))\stackrel{\delta_{a-1,b}}\longrightarrow H^{a}(G_K,\coker(\phi_{b-1})(s))\to E_{2,n}^{a,b}\to \\
		H^{a}(G_K,\ker(\phi_{b})(s))\stackrel{\delta_{a,b}}\longrightarrow H^{a}(G_K,\coker(\phi_{b-1})(s))\to 
		\end{aligned}
\end{equation}

  Since $HK^b_{c}(X_{n,C},b)$ is a Banach-Colmez space, one has that $H^a(G_K,HK^b_{c}(X_{n,C},b))$ is a finite dimensional $\bQ_p$-vector space. Because $\im(\phi_b)$ is a Banach-Colmez space (recall $\phi_b$ is induced by the Hyodo-Kato morphism and almost $C$-representations form an abelian category), it implies $H^a(G_K,\ker(\phi_b))$ is a finite dimensional $\bQ_p$-vector space.
  
   Now we show $\coker(\delta_{a-1,b})$ is a space of compact type. It is enough to show it for\\ $H^{a}(G_K,\coker(\phi_{b-1})(s))$. Since $DR_c(X_{n,C},b)$ is concentrated in degree betweem $d$ and $2d$, $H^{b-1}DR_c(X_{n,C},b)=0$ unless $b\geq d+1$ and we will assume it. Recall that 
$$\begin{aligned}DR_{c}(X_{n,C},b)=(H^d_{c}(X_n,\cO_X)\otimes^\square_K (\mathbf{B}^+_{\dR}/F^b)\to ...\to H^d_{c}(X_n,\Omega^{b-d-1})\otimes^\square_K (\mathbf{B}^+_{\dR}/F^{d+1})\to  \\H^d_{c}(X_n,\Omega^{b-d})\otimes^\square_K (\mathbf{B}^+_{\dR}/F^d)\to...\to H^d_c(X_n,\Omega^d)\otimes^\square_K \mathbf{B}^+_{\dR}/F^{b-d}\to  0)[-d].\end{aligned}$$

To simplify the notation, for any integer $k\geq 0$, we denote 
$$F^kK:=(H^d_c(X_{n},\Omega^{b-d-1})(s)/\im(d_c)(s))\otimes^\square_K (F^k\mathbf{B}^+_{\dR}/F^{d+1})$$
where $d_c$ denotes the differential operator in the compactly supported de Rham complex. It induces a decreasing filtration $F^k(\im(\phi_{b-1})(s))$ on $\im(\phi_{b-1})(s)$. Note that $F^k(\im\phi_{b-1}(s))$ is still an almost $C$-representation: It is the image of $ (H^{b-1}_{\HK,c}\otimes^\square_{\check{C}}  t^k \widehat{\mathbf{B}}^+_{\st})^{\varphi=p^b,N=0}$, which corresponds to global sections of the $\cO(-k)$-twist of the vector bundle associated to the $(\varphi,N)$-module $H^{b-1}_{\HK,c}(X_C)$, under the Hyodo-Kato morphism. Therefore, we have the following exact sequence
  \begin{equation}\label{formuladecomposecokernel}\begin{CD}
  	0\to  F^dK/F^d(\im(\phi_{b-1})(s))@>>>\coker(\phi_{b-1}(s))@>>>\displaystyle\frac{H_{\dR,c}^{b-1}(X_n)\otimes^\square_K (\mathbf{B}_{\dR}^+/F^d)}{\im(\phi_{b-1}(s))/F^d}\to 0
  \end{CD}\end{equation}
  
Since $H_{\dR,c}^{b-1}(X_n)$ is a finite dimensional $K$-vector space and $\im(\phi_{b-1}(s))/F^d$ is an almost $C$-representation, we see that Galois cohomology groups of the third term of the exact sequence \eqref{formuladecomposecokernel} are finite dimensional $\bQ_p$-vector spaces. Applying generalized Tate's formula (\cite[Proposition 3.15]{curves}) to $F^dK$, we get that $H^a(G_K,\coker(\phi_{b-1})(s))$ is a finite dimensional $\bQ_p$-vector space unless $b=j+d$ and $a=0,1$.
  
  In this case, since $H^a(G_K,F^d(\im(\phi_{b-1})(s)))$ are finite dimensional $\bQ_p$-vector spaces, to show $H^a(G_K,\coker(\phi_{b-1}(s)))$ is of compact type it suffices to show it for 
  $$H^a(G_K,H^d_c(X_n,\Omega^{b-d-1})(s)\otimes_K^\square C(d)/\im(d_c\otimes C(d)))\cong H^d_c(X_n,\Omega^{b-d-1})/\im(d_c).$$

Let $k\geq 0$ be an integer. It suffices to show that $\im(d_c)\subseteq H^d_c(X_n,\Omega^k)$ is a closed subspace, because then $H^d_c(X_n,\Omega^k)/\im(d_c)$  is a space of compact type quotiented by a closed subspace hence is a space of compact type. To show our claim, we consider the following classical (not condensed) spaces. There are strict exact sequences of nuclear Fr\'echet spaces where $d$ is the differential operator
   $$0\to \ker(d)\to H^0(X_n,\Omega^{d-k})\to \im(d)\to 0$$
   $$0\to \im(d)\to H^0(X_n,\Omega^{d-k+1})\to \coker(d)\to 0.$$
   Take strong duals of these two sequences and apply Serre duality (\cite{Chiarellotto1990}, \cite[Theorem 5.1]{achinger2025compactlysupportedpadicproetale}), we get strict exact sequences of spaces of compact type which factor the differential $d_c: H^d_c(X_n,\Omega^{k-1})\to H^d_c(X_n,\Omega^k)$ (The dual morphisms are still exact because the strong dual induces an anti-equivalence between nuclear Fr\'echet spaces and spaces of compact type \cite[Corollary 1.4]{schneider2002locally} and strictness of a map $f$ can be characterized categorically by $\im(f)\cong \coim(f)$). To see $\im(d_c)=\im(d)^*$ is a closed subspace, it suffices to observe that $\ker(d)^*$ is a Hausdorff space since $\ker(d)$ is a nuclear Fr\'echet space. \par 
    Therefore all spaces of $E_{2,n}^{a,b}$ are of compact type. From the description of the $E_3$-page, we find $H_{\proet,c}^i(X_n,\bQ_p(j))$ is finite dimensional or can be written as an extension of a space of compact type by a finite dimensional space (it splits using Lemma \ref{lemmacompacttypedual}), which is of compact type.\\  
    
To show that transition maps $f_n$ are compact maps. By Lemma \ref{lemmaextendcompact}, it suffices to show that transition maps between $E_{3,n}^{a,b}$ and $E_{3,n+1}^{a,b}$ are compact morphisms so it is enough to show it for the case $a=0,1$ and $b=j+d$. Use Lemma \ref{lemmaextendcompact} again, it suffices to show that the transition maps between $H^a(G_K,H^{b-1}DR_{c}(X_{n,C},b)(-d))$, when $n$ varies, are compact. Since $\im(d_c:H^d_c(X_n,\Omega^{j-2})\to H^d_c(X_n,\Omega^{j-1}))$ is a closed subspace, after applying \cite[Theorem 8.1.3(xi)]{perez2010locally}, it suffices to show transition maps between $H^d_{c}(X_{n},\Omega^j)$ are compact, after reindexing $j-1$ by $j$. 
For this, let $Y_n$ be a rigid analytic affinoid such that $X_n\subseteq Y_n \Subset X_{n+1}$. So there are morphisms 
$$\Omega^{d-j}(X_{n+1})\to \Omega^{d-j}(Y_n)\to \Omega^{d-j}(X_n).$$
Taking steorotype dual of the diagram above and use Lemma \ref{lemmasametopologyofdual}, we can applying Serre duality of Stein spaces to both ends of the diagram. Thus we get 
$$ H_c^d(X_n,\Omega^j)\stackrel{f_1}\longrightarrow (\Omega^{d-j}(Y_n))'_c\stackrel{f_2}\longrightarrow H_c^d(X_{n+1},\Omega^j) $$ 
which factors the transition map $f:H^d_{c}(X_{n},\Omega^j)\to H^d_{c}(X_{n+1},\Omega^j)$.
Because $\Omega^{d-j}(Y_n)$ is a Banach space, the stereotype dual is a Smith space hence is a compact object in the category of solid $\bQ_p$-vector spaces. Since $H_c^d(X_{n+1},\Omega^j)$ is of compact type, it can be written as an inductive limit of Banach spaces. Using compactness of $\Omega^{d-j}(Y_n)'_c$, $f_2$ factors through some Banach space $W$ (use the fully faithfulness of Proposition \ref{propositioncondensedbase}(4)) as the following: 
$$\Omega^{d-j}(Y_n)'_c \stackrel{f_2'}\longrightarrow W\longrightarrow H_c^d(X_{n+1},\Omega^j).$$ 
 We need to show $f_2'f_1$ is a compact morphism but this follows from the fact that $H_c^d(X_n,\Omega^j)$ is nuclear (cf. \cite[Proposition 19.5]{schneider2013nonarchimedean}).
 
\end{proof}

\begin{corollary} \label{corollarytopologydagger}
	Let $X$ be a geometrically connected smooth dagger affinoid variety over $K$. Then the cohomology group $H^i_{\proet}(X,\bQ_p(j))$ is a space of compact type.
\end{corollary}

\begin{proof}
	Since (use $\bQ_{p,\square}[G_K^i]$ are compact in the last quasi-isomorphism)
    $$\begin{aligned}
        &R\Gamma_{\proet}(X,\bQ_p(j))\simeq \colim_n R\Gamma_{\proet}(X_n^0,\bQ_p(j))\\
        &\simeq \colim_n R\Gamma(G_K,R\Gamma_{\proet}(X_{n,C}^0,\bQ_p(j)))\simeq R\Gamma(G_K,R\Gamma_{\proet}(X_C,\bQ_p(j)))
    \end{aligned}$$
    we have a Hochschild-Serre spectral sequence 
    $$E_2^{a,b}=H^a(G_K,H^b_{\proet}(X_C,\bQ_p(j)))\Rightarrow H_{\proet}^{a+b}(X,\bQ_p(j)).$$
    Note that Theorem \ref{theoremclassicalperiod}(1) also works for smooth dagger affinoid $X_C$ (\cite[Proposition 5.20, Theorem 6.9]{niziol2021cohomology}) and $H^i_{\dR}(X)$ and $H^i_{\HK}(X_C)$ are finite dimensional vector spaces. We can apply verbatim the same argument as in the proof of Theorem \ref{theorem1}, replacing $X_n$ there by $X$. It yields that $H_{\proet}^{a+b}(X,\bQ_p(j))$ admits a finite filtration whose graded pieces are spaces of compact type (use that $\Omega^i(X)/\ker(d)$ is of compact type). Moreover, except for one possible graded piece, all the others are finite dimensional $\bQ_p$-vector spaces. This implies $H_{\proet}^{a+b}(X,\bQ_p(j))$ is of compact type.
\end{proof}

\section{Arithmetic Duality}
Let $X$ be a smooth, geometrically connected Stein space of dimension $d$ over $K$. In this section, we deduce the duality for arithmetic pro-\'etale cohomology of $X$. We first follow \cite{geometricdual} to construct the syntomic complex with $G_K$-action on the Fargues-Fontaine curve and recall the main result in \cite{geometricdual}. Then we prove the local Tate duality of the syntomic sheaf and deduce the arithmetic duality.
 
 \subsection{Quasi-Coherent sheaves on the Fargues-Fontaine curve}
 
 \subsubsection{Generalities}
We recall the theory of quasi-coherent sheaves on analytic adic spaces developed in \cite{andreychev2021pseudocoherent}.

Let $Z$ be an analytic adic space over $\bQ_p$. We denote the $\infty$-category of solid quasi-coherent sheaves, defined in \cite{andreychev2021pseudocoherent}, via $D\qcoh(Z)$; the full $\infty$-subcategory of solid nuclear sheaves on $Z$ via $D\nuc(Z)$ and the full $\infty$-subcategory of perfect complexes, i.e complexes analytic locally quasi-isomorphic to a bounded complex of projective $\cO_Z(U)$-modules, via $\perf(Z)$ (cf. \cite[Sec. 5]{andreychev2021pseudocoherent}). When $Z=\Spa(R,R^+)$ is affinoid, we have 
\begin{equation}\label{formulaaffinoidequivalence}
	D\qcoh(Z)\simeq D((R,R^+)_\square)
\end{equation}
$$D\nuc(Z)\simeq D\nuc((R,R^+)_\square)$$
$$\perf(Z)\simeq \perf(R)$$
where  $D((R,R^+)_\square)$ is the derived $\infty$-category of solid $(R,R^+)_\square$-modules (\cite[Sec. 3.3]{andreychev2021pseudocoherent}) and $\perf(R)$ is the $\infty$-category of (classical) perfect $R$-modules. The categories $D\qcoh(Z)$, $D\nuc(Z)$ and $\perf(Z)$ are equipped with compatible symmetric monoidal structures and are all closed relative to it.

Note that (see \cite[Sec. 3.3]{andreychev2021pseudocoherent}) these categories $D\qcoh(Z)$, $D\nuc(Z)$ and $\perf(Z)$ can be defined when $Z=(R,R^+)$ is a pair such that $R$ is a complete Huber ring and $R^+ \subseteq R^\circ$ is an arbitrary subring.\\

 Suppose now $Z=\Spa(R,R^+)$ is affinoid such that $(R,R^+)$ is a complete analytic Huber pair and $\underline{R}$ is a solid nuclear $\bQ_p$-algebra. Then $(R,\bZ)_\square$ is an analytic ring (cf. \cite[Proposition. A.29]{bosco2021p}) together with a morphism between analytic rings $(R,\bZ)_\square\to (R,R^+)_\square$. This morphism induces a base change functor 
 $$-\otimes_{(R,\bZ)_\square} (R,R^+)_\square: D((R,\bZ)_\square)\to D((R,R^+)_\square) $$
 between derived $\infty$-category of solid modules.
 
  \begin{proposition}[{\cite[Theorem A.22]{bosco2023rational}}]\label{propositionnuclearanalticstructure}
 	The functor 
 	$$-\otimes_{(R,\bZ)_\square} (R,R^+)_\square: D\nuc((R,\bZ)_\square)\to D\nuc((R,R^+)_\square) $$
 	induces an equivalence between symmetric monoidal $\infty$-categories.
 \end{proposition}

 \begin{proof}
 First note that because $(R,R^+)_\square$ is normalized ($R$ is a complete $(R,R^+)_\square$-module), for any profinite set $S$, we have $\underline{\Hom}_{\bZ}(\bZ[S],R)=\underline{\Hom}_{R}((R,R^+)_\square[S],R)$ is already $(R,R^+)_\square$-complete.\par 
   Recall that both categories are full $\infty$-subcategories of $D\cond(R)$ and stable under colimits. By \cite[Proposition A.15]{bosco2023rational}, both of them are generated by objects of the form $\underline{\Hom}_R(R[S],R)$, for varying $S$ extremally disconnected, under shifts and colimits. So the base change functor induces identity map on these generators and we deduce that it is an equivalence of categories. Compatibility with symmetric monoidal structure follows because the base change functor is symmetric monoidal. 
 	\end{proof}

 \begin{remark}\label{lemmawellknown}
     Assume $M,N\in D((R,R^+)_\square)$ then $R\underline{\Hom}_{R}(M,N)\in D((R,R^+)_\square)$ (here $R\underline{\Hom}_{R}(-,-)$ is the internal Hom in the category of condensed $R$-modules) and the natural map 
     $$R\underline{\Hom}_{R}(M,N)\longrightarrow R\underline{\Hom}_{(R,R^+)_\square}(M,N) $$
     is a quasi-isomorphism in $D((R,R^+)_\square)$. This follows easily from a computation using adjunction and $-\otimes^L_{R} (R,R^+)_\square$ is symmetric monoidal.

     Similarly, the natural morphism
     $R\underline{\Hom}_{(R,\bZ)_\square}(M,N)\to  R\underline{\Hom}_{(R,R^+)_\square}(M,N)$ is also a quasi-isomorphism in $D((R,R^+)_\square)$.



 \end{remark}


\subsubsection{Quasi-coherent sheaves on the Fargues-Fontaine curve} 
$ $\\
We recall the basics of the theory of quasi-coherent sheaves on the Fargues-Fontaine curve. Let $S=\Spa(R,R^+)\to \Spa C^\flat$ be an affinoid perfectoid space in characteristic $p$, we define 
$$Y_{\FF,S}:=\Spa(W(R^+))\setminus V(p[p^\flat])=\cup_{I\subset (0,\infty)} Y_{\FF,S,I}$$
where $I=[a,b]\subset (0,\infty)$ is a closed interval with rational endpoints and $Y_{\FF,S,I}\subseteq Y_{\FF,S}$ is the open subset defined by 
$$Y_{\FF,S,I}:=\{|\cdot|:|p|^b\leq |[p^\flat]|\leq |p|^a\}\subseteq Y_{\FF,S}.$$ 
 Let $\FF_S:=Y_{\FF,S}/\varphi^\bZ$ be the Fargues-Fontaine curve. We denote $Y_{\FF,S,I}=\Spa(B_{S,I},B_{S,I}^+)$ where $B_{S,I}^+=W(R^+)\langle\frac{p}{[(p^{\flat})^{1/b}]},\frac{[(p^\flat)^{1/a}]}{p}\rangle$, $B_{S,I}=B_{S,I}^+[1/p]$ and $B_S:=\cO(Y_{\FF,S})=\lim_{I\subseteq (0,\infty)} B_{S,I}$. If $S=\Spa(C^\flat ,\cO_{C^\flat})$, we will omit $S$.\\
 
 \begin{definition}\label{definitiondqcoh}
 	We define the $\infty$-category of quasi-coherent $\varphi$-modules on $Y_{\FF,S}$ to be
		$$\xymatrix{
		{D\mathrm{QCoh}(Y_{\FF,S})^\varphi:=eq(D\mathrm{QCoh}(Y_{\FF,S})}\ar@<0.5ex>[r]^-{\varphi^*} \ar@<-0.5ex>[r]_-{id^*} & D\mathrm{QCoh}(Y_{\FF,S}))
		}$$
		whose objects are pairs $(\cE,\varphi_{\cE})$ where $\cE\in D\qcoh(Y_{\FF})$ and $\varphi_{\cE}:\varphi^*\cE\simeq \cE$. We define $D\nuc(Y_{S,\FF})^\varphi$ and $\perf(Y_{\FF})^\varphi$ similarly (using analytic descent of solid nuclear modules and perfect modules \cite[Theorem 5.42, 5.3]{andreychev2021pseudocoherent}).
 \end{definition}
 
 We will use the following convention: If $p\neq 2$ we will set $u=(p-1)/p$ and $v=p-1$ and if $p=2$ we will set $u=3/4$ and $v=3/2$\footnote{The $u,v$ are chosen such that $v=pu$, $t\in B_{S,[u,v]}$ and $t$ is a unit in $B_{S,[u,u]}$ and $B_{S,[v,v]}$.}. If $S$ is constructed as a tilt of a perfectoid space $S^\sharp$, the associated divisor of $S^\sharp$ in $Y_{\FF,S}$ is denoted by $y_{\infty}$. On $Y_{S,[u,v]}$, it is defined by a distinguished element $t$ which is a unit in $B_{S,[u,v/p]}$.\\ 
 
 Since one can construct $\FF_{S}$ as quotient of $Y_{\FF,S,[u,v]}$ via identification $\varphi:Y_{\FF,S,[u,u]}\stackrel{\simeq}\longrightarrow Y_{\FF,S,[v,v]}$. Let $j:Y_{\FF,S,[u,v/p]}\to Y_{\FF,S,[u,v]}$ denote the map induced by inclusion. By \eqref{formulaaffinoidequivalence} and analytic descent, we get 
 $$ \xymatrix{
 D\qcoh(Y_{\FF,S})^\varphi\simeq D(B_{S,\square})^\varphi:=eq(D((B_{S,[u,v]},B^+_{S,[u,v]})_\square)\ar@<0.5ex>[r]^-{\varphi^*} \ar@<-0.5ex>[r]_-{j^*} & D((B_{S,[u,v/p]},B^+_{S,[u,v/p]})_\square))
}$$
It is the $\infty$-category of $(M_{S,[u,v]},\varphi_M)$ where $M_{S,[u,v]}\in D((B_{S,[u,v]},B^+_{S,[u,v]})_\square)$ and $\varphi_M$ is a quasi-isomorphism of $(B_{S,[u,v/p]},B^+_{S,[u,v/p]})_\square$-modules 
$$\varphi_M: \varphi^*M_{S,[u,v]}\stackrel{\simeq}\longrightarrow M_{S,[u,v/p]}:=j^*M_{S,[u,v]}$$
where we abuse the notation to view the Frobenius $\varphi$ as the composition  
$$\varphi: B_{S,[u,v]}\stackrel{\varphi}\longrightarrow B_{S,[u/p,v/p]}\stackrel{j}\longrightarrow B_{S,[u,v/p]}.$$
Here $j$ is the map induced by inclusion $Y_{\FF,S,[u,v/p]}\to Y_{\FF,S,[u/p,v/p]}$.\\

One can define analogously 
$$\xymatrix{
D\nuc(Y_{\FF,S})^\varphi\simeq D\nuc(B_{S,\square})^\varphi:=eq(D\nuc((B_{S,[u,v]},B^+_{S,[u,v]})_\square)\ar@<0.5ex>[r]^-{\varphi^*} \ar@<-0.5ex>[r]_-{j^*} & D\nuc((B_{S,[u,v/p]},B^+_{S,[u,v/p]})_\square)
)}$$

\begin{remark}\label{remarkdqcoh}
We will be interested in nuclear modules on the Fargues-Fontaine curve. By Proposition \ref{propositionnuclearanalticstructure}, we will also consider the following variant:
$$\xymatrix{
D(B_S)^\varphi:=eq(D((B_{S,[u,v]},\bZ)_\square)\ar@<0.5ex>[r]^-{\varphi^*} \ar@<-0.5ex>[r]_-{j^*} & D((B_{S,[u,v/p]},\bZ)_\square)).}$$
$$\xymatrix{
D\nuc(B_S)^\varphi:=eq(D\nuc((B_{S,[u,v]},\bZ)_\square)\ar@<0.5ex>[r]^-{\varphi^*} \ar@<-0.5ex>[r]_-{j^*} & D\nuc((B_{S,[u,v/p]},\bZ)_\square)).
}$$
It follows from Proposition \ref{propositionnuclearanalticstructure} that the base change functor induces an equivalence of $\infty$-categories $D\nuc(B_S)^\varphi\simeq D\nuc(B_{S,\square})^\varphi$. The same holds for the perfect complexes. (But it does not hold for $D(B_S)^\varphi$)
\end{remark}

Since the action of $\varphi$ on $Y_{\FF,S}$ is free and properly discontinuous, it follows from analytic descent for quasi-coherent sheaves that we have an equivalence of $\infty$-categories
$$D\qcoh(Y_{\FF,S})^\varphi\simeq D\qcoh(Y_{\FF,S}/\varphi^\bZ)=D\qcoh(\FF_S).$$
Thus, we obtain an equivalence of $\infty$-categories 
$$\cE_{\FF,S}: D(B_{S,\square})^\varphi\stackrel{\simeq}\longrightarrow D\qcoh(\FF_S).$$
Restricting $\cE_{\FF,S}$ to nuclear complexes and by Proposition \ref{propositionnuclearanalticstructure}, we obtain an equivalence of $\infty$-categories
$$\cE_{\FF,S}: D\nuc(B_S)^\varphi\stackrel{\simeq}\longrightarrow D\nuc(\FF_S).$$

\begin{lemma}[{\cite[Proposition 2.13]{geometricdual}}]\label{lemmaglobalsection}
	Given $(M,\varphi_M)\in D\nuc(B)^\varphi$. Let $\varphi'_M$ be the composition $M_{[u,v]}\to M_{[v,v]}\cong M_{[u,u]}$ (for notations, see discussions below Definition \ref{definitiondqcoh}) where the last isomorphism is the composition
	$$\begin{CD}
		M_{[v,v]}@>id\otimes 1>\cong > M_{[v,v]}\otimes_{B_{[v,v]},\varphi}^\square B_{[u,u]}@>\varphi_M>\cong> M_{[u,u]}.
	\end{CD} $$ 
	There is a quasi-isomorphism in $D(\bQ_{p,\square})$
	$$R\Gamma(\FF,\cE_{\FF}(M))\simeq [M_{[u,v]}\stackrel{\varphi_M'-1}\longrightarrow M_{[u,u]}].$$	
\end{lemma}

\begin{proof}
	
	 Choose a rational number $w$ such that $u<w<v$. Then $\FF$ is covered by $Y_{[u,w]}$ and $Y_{[w,v]}$ and we can use the \v{C}ech nerve of $Y_{\FF,[u,w]}\coprod Y_{\FF,[w,v]}\to \FF$ to compute the cohomology.
	 Note that
	$$R\Gamma(Y_{\FF,[u,v]},\cE_{\FF}(M)|_{[u,v]}) =R\underline{\Hom}_{B_{[u,v]}}(B_{[u,v]},M_{[u,v]})=M_{[u,v]}.$$
  Since $(\bZ[T],\bZ)_\square\to \bZ[T]_\square$ is a steady localization,  \cite[Exercise 12.17]{scholze2019lectures} allows us to simplify the nerve and we get
  $$R\Gamma(\FF,\cE_{\FF}(M))=[M_{[u,w]}\oplus M_{[w,v]}\stackrel{f-g}\longrightarrow M_{[u,u]}\oplus M_{[w,w]}]$$
  where $f=(j,j): M_{[u,w]}\to M_{[u,u]}\oplus M_{[w,w]}$ is the canonical map (we abuse the notation, here $j$ are different maps and induced by canonical inclusions) and $g=(\varphi_M',j): M_{[w,v]}\to M_{[u,u]}\oplus M_{[w,w]}.$
  We conclude from the following fiber sequence, because $Y_{\FF,[u,v]}=Y_{\FF,[u,w]}\cup Y_{\FF,[w,v]}$
  $$ M_{[u,v]}\to M_{[u,w]}\oplus M_{[w,v]}\to M_{[w,w]}.$$
\end{proof}

 \subsubsection{Symmetric monoidal structure}
 We translate the closed symmetric monoidal structure on $D\qcoh(\FF_S)$ to the category $D(B_{S,\square})^\varphi$. 
 
  We first recall the following result: For any rational number $w$ such that $u\leq w\leq v$, we have $B_{S,[u,w]}=B_{S,[u,v]}\langle f\rangle$ where $f=p/([p^\flat]^{\frac{1}{w}})\in B_{S,[u,v]}$ and $B_{S,[w,v]}= B_{S,[u,v]}\langle g\rangle$ where $g=[p^\flat]^{\frac{1}{w}}/p$ . We write $B_{1}:=(B_{S,[u,v]},B^+_{S,[u,v]})_\square$ and $B_{3}:=(B_{S,[u,w]},B^+_{S,[u,w]})_\square$. By \cite[Proposition 4.11]{andreychev2021pseudocoherent}, we have 
$$(-)\otimes^L_{B_{1}} B_{3}\simeq (-)\otimes^L_{(\bZ[T],\bZ)_\square} \bZ[T]_\square$$
where the map $(\bZ[T],\bZ)_\square\to B_{1}$ is induced by $T\to f$ and $\bZ[T]_\square:=(\bZ[T],\bZ[T])_\square$. Meanwhile, by \cite[Proposition 3.12]{andreychev2021pseudocoherent}, for $M\in D((\bZ[T],\bZ)_\square)$, one has 
$$ M\otimes_{(\bZ[T],\bZ)_\square}^L \bZ[T]_\square\simeq R\underline{\Hom}_{R}(R_\infty/R,M)[1] $$
where $R=\bZ[T]$ and $R_\infty=\bZ((T^{-1}))$. A similar expression holds if we replace $B_{[u,v]}$ by $B_{[w,v]}$ and $B_{[u,w]}$ by $B_{[w,w]}$.

 To simplify the notation, we set $B_1:=(B_{S,[u,v]},B^+_{S,[u,v]})_\square$ and $B_2:=(B_{S,[u,u]},B^+_{S,[u,u]})_\square$. This tensor product on $D(B_{S,\square})^\varphi$, denoted by $(-)\otimes^L_{B_{S,\square}^{\FF}} (-)$ is induced from the one on $D(B_1)$. That is, given $(M,\varphi_M),(N,\varphi_N)\in D(B_{S,\square})^\varphi$ their tensor product is defined by  
 $$\begin{aligned}
 &M\otimes^L_{B_{S,\square}^{\FF}} N:=(M\otimes^L_{B_1} N,\varphi_{M\otimes N})\\
 &\varphi_{M\otimes N}=\varphi_M\otimes\varphi_N: (M\otimes^L_{B_1} N)\otimes^L_{B_1,\varphi} B_2\stackrel{\simeq}\to (M\otimes_{B_1}^L N)\otimes^L_{B_1} B_2=(M_{S,[u,v/p]}\otimes_{B_2}^L N_{S,[u,v/p]})
 \end{aligned}$$
 Using Lemma \ref{lemmainternalHom}, the internal Hom in $D(B_{S,\square})^\varphi$, denoted by $R\sH om_{B_{S,\square}^{\FF}}(-,-)$, is defined as 
 \begin{equation}\label{interHomformula}\begin{aligned}
      R\sH om_{B_{S,\square}^{\FF}}(M,N)&:=(R\underline{\Hom}_{B_1}(M,N),\varphi_{M,N})\\
      \varphi_{M,N}:=(\varphi_M^{-1},\varphi_N) &: R\underline{\Hom}_{B_1}(M,N)\otimes^L_{B_1,\varphi} B_2\to R\underline{\Hom}_{B_1}(M,N)\otimes_{B_1}^L B_2.
\end{aligned}
 \end{equation}

 \begin{lemma}[{\cite[Lemma 2.4]{geometricdual}}] \label{lemmainternalHom} The canonical maps 
 $$R\underline{\Hom}_{B_1}(M,N)\otimes_{B_1,\varphi}^L B_2\to R\underline{\Hom}_{B_2}(M\otimes_{B_1,\varphi}^L B_2, N\otimes_{B_1,\varphi}^L B_2) $$
 $$R\underline{\Hom}_{B_1}(M,N)\otimes_{B_1}^L B_2\to R\underline{\Hom}_{B_2}(M\otimes_{B_1}^L B_2,N\otimes_{B_1}^L B_2)$$
 are quasi-isomorphisms.
 \end{lemma}
 
 \begin{proof}
 	The first map is induced by composition $B_{[u,v]}\stackrel{\varphi}\longrightarrow B_{[u/p,v/p]}\longrightarrow B_{[u,v/p]}$
 	where $\varphi$ is an isomorphism. Thus it suffices to show the second quasi-isomorphism. 
 	
 	Write $M=\colim_i M_i$ as colimit of compact projective objects $\{M_i\}_i$. Thus 
 	$$R\underline{\Hom}_{B_1}(M,N)=R\lim_i R\underline{\Hom}_{B_1}(M_i,N)$$
 	$$R\underline{\Hom}_{B_2}(M\otimes_{B_1}^L B_2,N\otimes_{B_1}^L B_2)=R\lim_i R\underline{\Hom}_{B_2}(M_i\otimes_{B_1}^L B_2,N\otimes_{B_1}^L B_2).$$
 	By \cite[Prop. 5.38]{andreychev2021pseudocoherent}, we have 
 	$$R\underline{\Hom}_{B_1}(M_i,N)\otimes^L_{B_1} B_2\simeq 
 	R\underline{\Hom}_{B_2}(M_i\otimes_{B_1}^L B_2,N\otimes^L_{B_1} B_2).$$
 	So it suffices to show $(-)\otimes_{B_1}^L B_2$ commutes with derived limit. But as we recalled above 
 	$$(-)\otimes_{B_1}^L B_2\simeq (-)\otimes^L_{(\bZ[T],\bZ)_\square} \bZ[T]_\square\simeq R\underline{\Hom}_R(R_\infty/R,-)[1]$$ 
 	which commutes with limits.
 \end{proof}
 

 We will also consider the tensor product and internal Hom on the category $D(B_S)^\varphi$. Similar to $D(B_{S,\square})^\varphi$, the tensor product on $D(B_S)^\varphi$, denoted by $(-)\otimes_{B^{\FF}_S} (-)$, is induced from the one on $D((B_{S,[u,v]},\bZ)_\square)$. By the adjoint functor theorem, there exists an internal Hom $R\sH om_{B^{\FF}_S}(-,-)$ on $D(B_S)^\varphi$, but we do not know how to describe it in general. This is because Lemma \ref{lemmainternalHom} will not hold in general if we replace $B_1$ and $B_2$ by $(B_{[u,v]},\bZ)_\square$ and $(B_{[u,v/p]},\bZ)_\square$ respectively. However, it holds in the following situation, in which case we can describe $R\sH om_{B^{\FF}_S}(-,-)$ analogously to (and compatible with) \eqref{interHomformula}. We will only consider internal Homs in $D\nuc(B_S)^\varphi$ under this assumption.

 \begin{lemma} \label{internalHomhypothesis}
     Write $B_1':=(B_{S,[u,v]},\bZ)_\square, B_2':=(B_{S,[u,v/p]},\bZ)_\square$. Assume we have $M=(M_1,\varphi_M),N=(N_1,\varphi_N)\in D\nuc(B_S)^\varphi$ such that 
     $$R\underline{\Hom}_{B_1'}(M_1,N_1)\in D\nuc(B'_1).$$ 
      Then the natural maps
      $$R\underline{\Hom}_{B_1'}(M_1 ,N_1) \otimes^L_{B'_1,\varphi} B_2'\longrightarrow R\underline{\Hom}_{B_2'}(M_1 \otimes^L_{B'_1,\varphi} B_2',N_1\otimes^L_{B'_1,\varphi} B_2')$$
     $$R\underline{\Hom}_{B_1'}(M_1 ,N_1) \otimes^L_{B'_1} B_2'\longrightarrow R\underline{\Hom}_{B_2'}(M_1 \otimes^L_{B'_1} B_2',N_1\otimes^L_{B'_1} B_2')$$
     are quasi-isomorphisms.
 \end{lemma}

 \begin{proof}
     By an argument similar to Lemma \ref{lemmainternalHom}, it suffices to show the second map is a quasi-isomorphism. We still write $B_1:=(B_{S,[u,v]},B^+_{S,[u,v]})_\square$ and $B_2:=(B_{S,[u,u]},B^+_{S,[u,u]})_\square$.\par
     Note that if $K\in D\nuc(B_1')$ then $K\otimes^L_{B_1'} B_2' \in D\nuc(B_2')$, since nuclear objects are preserved under base change between analytic rings (\cite[Corollary A.12]{bosco2023rational}). In particular, applying Proposition \ref{propositionnuclearanalticstructure}, we have that $K$ is  $B_1$-complete, $K\otimes^L_{B_1'} B_2'$ is  $B_2$-complete and 
     \begin{equation}\label{formulabc}
         K\otimes^L_{B_1'} B_2'\simeq K\otimes^L_{B_1'} B_2 \simeq  K\otimes^L_{B_1} B_2.
     \end{equation}
      Now we have the following computation:
     $$\begin{aligned}
         R\underline{\Hom}_{B_2'}(M_1 \otimes^L_{B'_1} B_2',N_1\otimes^L_{B'_1} B_2')&\simeq R\underline{\Hom}_{B_2}(M_1 \otimes^L_{B_1} B_2,N_1\otimes^L_{B_1} B_2)\\
         &\simeq R\underline{\Hom}_{B_1}(M_1 ,N_1) \otimes^L_{B_1} B_2\\
         &\simeq R\underline{\Hom}_{B'_1}(M_1 ,N_1) \otimes^L_{B'_1} B'_2
     \end{aligned}$$
     where the first isomorphism uses \eqref{formulabc} and Remark \ref{lemmawellknown}, the second one uses Lemma \ref{lemmainternalHom}, the third one uses the assumption $R\underline{\Hom}_{B_1'}(M_1,N_1)\in D\nuc(B'_1)$ and \eqref{formulabc}.
     
 \end{proof}

 \subsubsection{Galois equivariant nuclear sheaves}
 In the following, we will also need to keep track of the Galois action on certain nuclear sheaves. So we set $S=C^\flat$ to only consider the absolute Fargues-Fontaine curve. In presence of Proposition \ref{propositionnuclearanalticstructure}, we only need to consider analytic rings of the form $(B,\bZ)_\square$ (so we do not need to worry about the Galois action on the analytic structure). 
 
 For each $I\subseteq (0,\infty)$ a closed interval with rational endpoints, the absolute Galois group $G_K$ acts continuously on $B_I$ in a way which respects the multiplication on $B_I$ and acts trivially on $\bQ_p\subseteq B_I$. Thus we can form the category $\solid_{B_I}^{G_K}$ of solid $B_I$-modules with semilinear $G_K$-actions. More precisely, an object in $\solid_{B_I}^{G_K}$ is a pair $(M,\rho_M)$ where $M\in \Mod_{B_{I,\square}}^{\solid}$ (here $B_{I,\square}=(B_I,\bZ)_\square$) and $\rho_M: \bQ_{p,\square}[G_K]\otimes_{\bQ_p}^{\square} M\to M$ is a $\bQ_p$-linear map such that it is $B_I$-semilinear (satisfying an obvious diagram). A morphism $(M_1,\rho_1)\to (M_2,\rho_2)$ is a $B_I$-linear map $M_1\to M_2$ which is equivariant with respect to the actions $\rho_1$ and $\rho_2$. It is easy to check that this is a Grothendieck abelian category with a tensor product and an internal Hom similar to $G_K$-representations. We have the derived $\infty$-category $D(B_{I,\square})^{G_K}$ and the full $\infty$-subcategory of solidly nuclear modules $D\nuc(B_{I,\square})^{G_K}$ which contains those objects whose underlying complex is solidly nuclear.
 
 \begin{definition}
 	We define the ($\infty$-)category of $G_K$-equivariant nuclear sheaves to be 
 	$$\xymatrix{
 	D\nuc(B)^{\varphi,G_K}:=eq(D\nuc(B_{[u,v],\square})^{G_K}\ar@<0.5ex>[r]^-{\varphi^*} \ar@<-0.5ex>[r]_-{j^*}& D\nuc(B_{[u,v/p],\square})^{G_K}) 
 	}$$ where $B_{I,\square}:=(B_I,\bZ)_\square$. 
 	It is the $\infty$-category of $((M_{[u,v]},\rho_M),\varphi_M)$ where $(M_{[u,v]},\rho_M)\in D\nuc(B_{[u,v],\square})^{G_K}$ and $\varphi_M$ is a quasi-isomorphism in $D\nuc(B_{[u,v/p],\square})^{G_K}$
 	$$\varphi_M: \varphi^*M_{[u,v]}\stackrel{\simeq}\longrightarrow M_{[u,v/p]}.$$ 
 	Note that the $G_K$-action on $B_I$ commutes with the Frobenius action, thus we can form the base change $\varphi^*$.
 \end{definition}
 
 \begin{remark}
 	We can also define the derived ($\infty$)-category of solid modules
 	$$\xymatrix{
 	D(B)^{\varphi,G_K}:=eq(D(B_{[u,v],\square})^{G_K}\ar@<0.5ex>[r]^-{\varphi^*} \ar@<-0.5ex>[r]_-{j^*}& D(B_{[u,v/p],\square})^{G_K}) 
 	}$$
 	but in general they do not correspond to $G_K$-equivariant sheaves on the curve.
 \end{remark}
 
 The forgetful functor $F: D\nuc(B_{I,\square})^{G_K}\to D\nuc(B_{I,\square})$ induces a forgetful functor 
 $$F: D\nuc(B)^{\varphi,G_K}\to D\nuc(B)^{\varphi};\ (M,\rho_M,\varphi_M)\mapsto (M,\varphi_M)$$
  which commutes with all limits and colimits. By abuse of notation, we set
  $$\cE_{\FF}((M,\rho_M,\varphi_M)):=\cE_{\FF}(F(M,\rho_M,\varphi_M)).$$
 We refer to morphisms in $D\nuc(B)^{\varphi,G_K}$ as $G_K$-equivariant morphisms between the corresponding sheaves. By Lemma \ref{lemmaglobalsection}, and since the Frobenius commutes with the Galois action, it induces a Galois action on global sections. Thus we obtain a functor:
  $$R\Gamma(\FF,\cE_{\FF}(-)): D\nuc(B)^{\varphi,G_K}\longrightarrow D(\bQ_{p,\square}[G_K]).$$
  
  Assume $(M_1,\rho_1,\varphi_1),(M_2,\rho_2,\varphi_2)\in D\nuc(B)^{\varphi,G_K}$. The category $D\nuc(B)^{\varphi,G_K}$ also carries a symmetric monoidal structure induced from the one on $D\nuc(B_I)^{G_K}$ and $D\nuc(B)^{\varphi}$ which is given by 
  $$(M_1,\rho_1,\varphi_1)\otimes^L_{\FF,G_K} (M_2,\rho_2,\varphi_2)=(M_1\otimes_{B_{[u,v],\square}}^L M_2, \rho_1\otimes \rho_2
  ,\varphi_1\otimes\varphi_2).$$
  \indent By Lemma \ref{internalHomhypothesis}, we can describe the internal Hom under the hypothesis that 
  \begin{equation}\label{conditioninthom}
       R\underline{\Hom}_{B_{[u,v]}}(M_1,M_2)\in D\nuc(B_{[u,v],\square}).
  \end{equation}
  In fact, we need to see the natural $G_K$-equivariant map
  $$R\underline{\Hom}_{B_{[u,v]}}(M_1,M_2)\otimes^L_{B_{[u,v],\square}} B_{[u,v/p],\square}\to R\underline{\Hom}_{B_{[u,v/p]}}(M_1\otimes^L_{B_{[u,v],\square}} B_{[u,v/p],\square},M_2\otimes^L_{B_{[u,v],\square}} B_{[u,v/p],\square})  $$
  is a quasi-isomorphism. This can be checked after forgetting the $G_K$-action which then follows from Lemma \ref{internalHomhypothesis}. In this case, the internal Hom is given by
  \begin{equation}\label{formulaactiononHom}
  R\sH om_{\FF,G_K}((M_1,\rho_1,\varphi_1),(M_2,\rho_2,\varphi_2))=(R\sH om_{\FF}(M_1,M_2),\rho_2\rho_1^{-1},\varphi_{M_1,M_2}).
  \end{equation}
  \indent Since we will always remember the Galois action, we will abuse the notation and write $(-)\otimes^L_{\FF,G_K}(-)$ as $(-)\otimes^L_{\FF}(-)$, $R\sH om_{\FF,G_K}(-,-)$ as $R\sH om_{\FF}(-,-)$ and $R\Gamma(\FF,R\sH om_{\FF,G_K}(-,-))$ as $R\Hom_{\FF}(-,-)$.

\subsection{Syntomic complexes and duality on the Fargues-Fontaine curve}
$ $

Let $X$ be a smooth rigid analytic Stein variety over $K$ of dimension $d$.

\subsubsection{The de Rham sheaf}
Recall that because of acyclicity of coherent cohomology on Stein space, the filtered $\mathbf{B}_{\dR}^+$-cohomology (and compactly supported version) is represented by the complex (if $a<0$, $t^aB_{\dR}^+:=F^aB_{\dR}^+=B_{\dR}^+$)
$$F^rR\Gamma(X_C/\mathbf{B}_{\dR}^+)\simeq \cO(X)\otimes_K^\square t^r\mathbf{B}_{\dR}^+\to \Omega^1(X)\otimes_K^\square t^{r-1}\mathbf{B}_{\dR}^+\to \cdots \Omega^d(X)\otimes_K^\square t^{r-d}\mathbf{B}_{\dR}^+. $$
 $$F^rR\Gamma_{c}(X_C/B_{\dR}^+)\simeq (H^d_c(X,\cO)\otimes_K^\square t^r\mathbf{B}^+_{\dR}\to\cdots\to H^d_c(X,\Omega^d)\otimes_K^\square t^{r-d}\mathbf{B}_{\dR}^+)[-d].$$
 Therefore, for each $r\in \bN$, the truncated $\mathbf{B}_{\dR}^+$-complex $DR(X_C,r)$ and $DR_c(X_C,r)$ are defined as in Remark \ref{remarkonderhamcomplex}.

For $*\in\{\emptyset,c\}$, we associate a quasi-coherent sheaf to $DR_*(X_C,r)$.  Define 
 $$R\Gamma_{\dR,*,B_{[u,v]}}(X_C,r)$$
 via replacing $B_{\dR}^+$ in \eqref{formulaofDRcomplex} by $B_{[u,v]}$. We view $\Omega^i(X)$ and $H^d_c(X,\Omega^i)$ as $K$-vector spaces with trivial $G_K$-action. In fact, because $B_{[u,v]}/t^a\cong B_{dR}^+/t^a$, this does not change the complex.  Since $\Omega^i(X)$ and $H^d_c(X,\Omega^i)$ are solid nuclear $K$-vector spaces and solid nuclearity is stable under base change (\cite[Corollary A.12]{bosco2023rational}), $R\Gamma_{\dR,*,B_{[u,v]}}(X_C,r)$ is a solidly nuclear $B_{[u,v],\square}$-complex. Recall that since $t$ is a unit in $B_{[u,u]}$, we have  
 $$R\Gamma_{\dR,*,B_{[u,v]}}(X_C,r)\otimes_{B_{[u,v],\square}}^L B_{[u,u]}=0.$$
 Therefore, the triple 
 $$R\Gamma^{B,\cE}_{\dR,*}(X_C,r):= (R\Gamma_{\dR,*,B_{[u,v]}}(X,r),\rho_{B_{[u,v]}},0)$$
 lies in $ D\nuc(B)^{\varphi,G_K}$ with the corresponding nuclear sheaf  $$\cE_{\dR,*}(X_C,r):=\cE_{\FF}(R\Gamma^{B,\cE}_{\dR,*}(X_C,r)).$$
 We set $R\Gamma_{\dR,*}^B(X_C,r):=R\Gamma(\FF,\cE_{\dR,*}(X_C,r))\in D(\bQ_{p,\square}[G_K])$.\\
 
Let $i_\infty:\Spa C\to \FF$ be the closed immersion corresponding to the untilt $C$ of $C^\flat$. The following lemma is obvious from the construction:
 \begin{lemma}
 	For $r\in \bN$, there is a natural quasi-isomorphism
 	$$\cE_{\dR,*}(X_C,r)\simeq i_{\infty,*}R\Gamma_{\dR,*}(X_C,r).$$
 \end{lemma}

\subsubsection{The Hyodo-Kato sheaf} Recall that 
  $$B_{\log}:=B\otimes_{\mathrm{Sym}_{\bZ}\cO_{C^\flat}^\times} \mathrm{Sym}_{\bZ} C^{\flat,\times}$$ is defined in \cite[Sec. 10.3.1]{fargues2019courbes}. It is non-canonically isomorphic to $B[U]$, endowed with actions of $G_K,\varphi$ and $N$ which is compatible with $\varphi,G_K$-action on $B$. For $g\in G_K$, the actions are induced by
 $$ g(U)=U+\log([g(p^\flat)/p^\flat]);\ \varphi(U)=pU;\ N=\frac{d}{dU}.$$
 In particular, $N\varphi=p\varphi N$ and the $G_K$-action commutes with $\varphi$ and $N$. There is a $G_K$-equivariant embedding $B_{\log}\to B_{\dR}^+$ induced by $U\to \log([p^\flat]/p])$.\\
 For $I=[a,b]\subset (0,\infty)$ a closed interval with rational endpoints, we have similarly
 $$B_{I,\log}:=B_I\otimes_{\mathrm{Sym}_{\bZ}O_{C^\flat}^\times} \mathrm{Sym}_{\bZ} C^{\flat,\times}$$
 endowed with the action of $G_K$ and $N$ and $G_K$-equivariant embedding into $B_{\dR}^+$. We view them as a condensed $\bQ_p$-vector space and note that it coincides with the period sheaf definition in \cite[Notation 2.31]{bosco2023rational} using \cite[Corollary 4.9]{bosco2021p} (See also \cite[Example 4.10]{bosco2021p}).\\

 Let $X$ be a smooth rigid analytic variety over $K$, $r$ be an integer and $*\in\{\emptyset,c\}$. Let $R\Gamma_{\HK,*}(X_C)\in D_{\varphi,N,G_K}(\check{C})$ be the (compactly supported) Hyodo-Kato cohomology defined in \cite[Sec. 4]{niziol2021cohomology} and \cite[Sec. 3.2]{achinger2025compactlysupportedpadicproetale}, where $D_{\varphi,N,G_K}(\check{C})$ is the derived $\infty$-category of solid $(\varphi,N,G_K)$-modules over $\check{C}$ (\cite[Sec. 3]{bosco2023rational}).

 \begin{definition} Let $X$ be a smooth Stein space over $K$ and $*\in\{\emptyset,c\}$, for $I=[a,b]\subset (0,\infty)$ a closed interval with rational endpoint, we define
 		$$R\Gamma_{\HK,*,B_I}(X_C,r):=(R\Gamma_{\HK,*}(X_C)\{r\}\otimes^{L_\square}_{\check{C}} B_{I,\log})^{N=0}$$
 	where the twist $\{r\}$ means that the Frobenius on $R\Gamma_{\HK,*}(X_C)$ is the usual Frobenius action divided by $p^r$. 
 \end{definition}

 \begin{lemma}\label{lemmaHKisnuclear}
 	If $X$ is a smooth Stein space over $K$ and $I\subseteq J\subseteq (0,\infty)$ are two closed intervals with rational endpoints, then
 	$$R\Gamma_{\HK,*,B_I}(X_C,r)=R\Gamma_{\HK,*,B_J}(X_C,r)\otimes^{L_\square}_{B_{J}} B_I $$ 
 	$$R\Gamma_{\HK,*,B_I}(X_C,r)\in D\nuc(B_I)$$
 	
  \end{lemma}
 
 \begin{proof} 
 	The first claim follows from the following computation:
 	$$\begin{aligned}
 		&R\Gamma_{\HK,*,B_I}(X_C,r)=(R\Gamma_{\HK,*}(X_C)\{r\}\otimes^{L_\square}_{\check{C}} B_{I,\log})^{N=0}\\
 		&=(R\Gamma_{\HK,*}(X_C)\{r\}\otimes^{L_\square}_{\check{C}} (B_{J,\log}\otimes^{L_\square}_{B_{J}} B_I))^{N=0}\\
 		&=(R\Gamma_{\HK,*}(X_C)\{r\}\otimes^{L_\square}_{\check{C}} B_{J,\log})^{N=0}\otimes^{L_\square}_{B_{J}} B_I=R\Gamma_{\HK,*,B_J}(X_C,r)\otimes^{L_\square}_{B_{J}} B_I
 	\end{aligned}$$
 	where the third equality holds because $B_I$ is annihilated by $N$.
    
 	 For the second claim, since nuclearity is stable under base change and can be checked on cohomology groups (\cite[Theorem A.17(iii)]{bosco2023rational}), 	it follows from that $R\Gamma_{\HK,*}(X_C)\in D\nuc(\check{C})$ and that monodromy operator is $B_I$-linear.
 \end{proof}

Now we define the triple
$$R\Gamma_{\HK,*}^{B,\cE}(X_C,r):=(R\Gamma_{\HK,*,B_{[u,v]}}(X_C,r),\rho_{\HK,*}\otimes \rho_{B_{[u,v]}},\varphi_{\HK,*}\{r\}\otimes\varphi) $$
where 
$$\rho_{\HK,*}\otimes\rho_{B_{[u,v]}}: \bQ_{p,\square}[G_K]\otimes^{L_\square}_{\bQ_{p}} R\Gamma_{\HK,*,B_{[u,v]}}(X_C,r)\to R\Gamma_{\HK,*,B_{[u,v]}}(X_C,r) $$
is induced from the $G_K$-action on Hyodo-Kato cohomology and $B_{[u,v]}$ and
$$\varphi_{\HK,*}\{r\}\otimes \varphi: R\Gamma_{\HK,*,B_{[u,v]}}(X_C,r)\otimes_{B_{[u,v]},\varphi}^{L_\square} B_{[u,v/p]}\stackrel{\simeq}\longrightarrow R\Gamma_{\HK,*,B_{[u,v/p]}}(X_C,r) $$ 
is induced from the Frobenius on the Hyodo-Kato cohomology and $\varphi:B_{[pu,v]}\stackrel{\simeq}\to B_{[u,v/p]}$. \par
The triple $R\Gamma_{\HK,*}^{B,\cE}(X_C,r)$ defines an object in $D\nuc(B)^{\varphi,G_K}$ and we get the associated Hyodo-Kato sheaf 
$$\cE_{\HK,*}(X_C,r):=\cE_{\FF}(R\Gamma_{\HK,*}^{B,\cE}(X_C,r)).$$
We set $R\Gamma_{\HK,*}^B(X_C,r):=R\Gamma(\FF,\cE_{\HK,*}(X_C,r))\in D(\bQ_{p,\square}[G_K])$.

\begin{remark}\label{remarkslope}
Recall that if $D$ is a finite dimensional $(\varphi,N)$-module over $\check{C}$, we have an exact sequence 
 $$0\to (D\otimes_{\check{C}} B)\stackrel{\beta}\longrightarrow (D\otimes_{\check{C}} B_{\log})\stackrel{N}\longrightarrow (D\otimes_{\check{C}} B_{\log})\to 0
 $$
where $\beta$ is the $\varphi$-equivariant isomorphism (\cite[Sec. 10.3.2]{fargues2019courbes})
 $$ \beta: D\otimes_{\check{C}} B\to (D\otimes_{\check{C}} B_{\log})^{N=0};\ y\mapsto \sum_{i\geq 0} \frac{(-1)^i}{i!} N^i(y)X^i.$$ The right exactness can be checked easily by induction on $m$ such that $N^m=0$.

 Assume $X$ is a smooth Stein space over $K$ geometrically connected of dimension $d$ and $\{X_n\}$ is a covering of $X$ by adapted naive interiors of smooth dagger affinoids. Then $H^i_{\HK}(X_{n,C})$ and $H^i_{\HK,c}(X_{n,C})$ are finite dimensional. The underlying complex of $(R\Gamma_{\HK,c}(X_{n,C})\{r\}\otimes_{\check{C}}^{L_\square} B_{[u,v],\log})^{N=0}$ is a perfect complex over $B_{[u,v],\square}$ so we can form the cohomology sheaf $H^i(\cE_{\HK,*}(X_{n,C},r))$.   By duality of Hyodo-Kato cohomology, the Frobenius slopes of $H^i_{\HK}(X_{n,C})$ and $H^i_{\HK,c}(X_{n,C})$ are non-negative and $\leq d$. If $r\geq d$, we find $H^i(\cE_{\HK,*}(X_{n,C},r))$ are vector bundles of non-negative slopes on $\FF$ (here one uses the $\varphi$-equivariant isomorphism $\beta$ in the last paragraph).
 \end{remark}

\subsubsection{Hyodo-Kato morphism}

We lift the Hyodo-Kato morphism to these sheaves. Let $*\in\{\emptyset,c\}$, and $r\in \bN$. For a smooth Stein space $X$ over $K$, recall that we have $G_K$-equivariant Hyodo-Kato isomorphisms (\cite[Sec. 4]{niziol2021cohomology}, \cite[Sec. 3.2]{achinger2025compactlysupportedpadicproetale})
$$\iota_{\HK,*}: R\Gamma_{\HK,*}(X_C)\otimes_{\check{C}}^{L_\square} \mathbf{B}_{\dR}^+\stackrel{\simeq}\longrightarrow R\Gamma_{*}(X_C/\mathbf{B}_{\dR}^+)\simeq R\Gamma_{\dR,*}(X)\otimes_K^{L_\square} \mathbf{B}_{\dR}^+.$$  
Combine it with $B_{[u,v],log}\to B_{[u,v]}/t^r=\mathbf{B}_{\dR}^+/t^r$, it defines a map of solid $B_{[u,v]}$-modules
$$\iota_{\HK,*}: R\Gamma_{\HK,*,B_{[u,v]}}(X_C,r)\to R\Gamma_{\dR,*,B_{[u,v]}}(X_C,r) $$
such that the following diagram commutes
$$\xymatrix{
	R\Gamma_{\HK,*,B_{[u,v]}}(X_C,r)\ar^{\iota_{\HK,*}}[r] \ar_{\varphi_{\HK,*}\otimes \varphi}[d] & R\Gamma_{\dR,*,B_{[u,v]}}(X_C,r)\ar[d]\\
	 R\Gamma_{\HK,*,B_{[u,v/p]}}(X_C,r)\ar[r] & 0.
}$$
Therefore, it defines a map in $D\nuc(B)^{\varphi,G_K}$
$$\iota_{\HK,*}: R\Gamma_{\HK,*}^{B,\cE}(X_C,r)\to R\Gamma_{\dR,*}^{B,\cE}(X_C,r)$$
and induces a morphism between nuclear sheaves
$$\iota_{\HK,*}: \cE_{\HK,*}(X_C,r)\to \cE_{\dR,*}(X_C,r).$$

\subsubsection{Syntomic complexes}
Now we glue the Hyodo-Kato complex and de Rham complex together to get the syntomic complex.
\begin{definition} Let $*\in\{\emptyset,c\}$ and $r\in \bN$. Assume $X$ is a smooth connected Stein space over $K$. We set
  $$R\Gamma_{\syn,*}^{B,\cE}(X_C,r):=[R\Gamma_{\HK,*}^{B,\cE}(X_C,r)\stackrel{\iota_{\HK,*}}\longrightarrow R\Gamma_{\dR,*}^{B,\cE}(X_C,r)].$$
By \cite[Lemma 5.40]{andreychev2021pseudocoherent}, one sees that it defines an object in $D\nuc(B)^{\varphi,G_K}$ and we call it the (compactly supported) syntomic module. The corresponding syntomic sheaf is 
  $$\cE_{\syn,*}(X_C,r):=\cE_{\FF}(R\Gamma_{\syn,*}^{B,\cE}(X_C,r)).$$
  We put $R\Gamma^B_{\syn,*}(X_C,\bQ_p(r)):=R\Gamma(\FF,\cE_{\syn,*}(X_C,r))\in D(\bQ_{p,\square}[G_K])$.
\end{definition}

\begin{lemma}\label{sheavesaregood}
    Let $\cM_r\in D\nuc(B)^{\varphi,G_K}$ be one of $\cE_{\syn,c}(X_C,r), \cE_{\HK,c}(X_C,r), \cE_{\dR,c}(X_C,r)$ and $M_r\in D\nuc(B_{[u,v],\square})$ be the underlying $B_{[u,v],\square}$-module of $\cM_r$. Then $$R\underline{\Hom}_{B_{[u,v]}}(M_r,B_{[u,v]})\in D\nuc(B_{[u,v],\square}).$$
\end{lemma}

\begin{proof}
    The claim for $\cE_{\syn,c}(X_C,r)$ will follow from the claim for the other two.\\
    For $\cE_{\HK,c}(X_C,r)$, it suffices to check nuclearity on cohomology sheaves. After choosing a covering by adapted naive interior, the result follows from that the cohomology is a sequential inverse limit of finite projective $B_{[u,v]}$-modules and nuclear $B_{[u,v],\square}$-modules are stable under sequential inverse limit (\cite[Theorem A.43]{bosco2021p}).\par
    For $\cE_{\dR,c}(X_C,r)$, it follows from the explicit expression of $R\Gamma_{\dR,c,B_{[u,v]}}(X_C,r)$ and that $R\underline{\Hom}_K(H^d_c(X,\Omega^i),K)\simeq \Omega^{d-i}(X)[0]$ is solidly nuclear. 
\end{proof}

Via a change of period ring argument, one deduces the following result from classical period isomorphisms (cf. \cite[Theorem 6.9]{niziol2021cohomology}, \cite[Theorem 6.13]{achinger2025compactlysupportedpadicproetale}).

\begin{proposition}[{\cite[Theorem 3.9, Lemma 3.15]{geometricdual}}]\label{propositionperiodisomorphism}
	Let $*\in\{\emptyset,c\}$ and $X$ be a smooth Stein space over $K$. There is a $G_K$-equivariant quasi-isomorphism in $D(\bQ_{p,\square})$
	$$\tau^{\leq r}R\Gamma^B_{\syn,*}(X_C,\bQ_p(r))\stackrel{\simeq}\longrightarrow \tau^{\leq r}R\Gamma_{\proet,*}(X_C,\bQ_p(r)). $$
\end{proposition}

\subsubsection{Poincar\'e duality of syntomic sheaves}
One of the main results of \cite{geometricdual} is the following Poincar\'e duality of syntomic sheaves on the Fargues-Fontaine curve, which lifts the de Rham and Hyodo-Kato dualities to sheaves on the curve.

\begin{theorem}[{\cite[Theorem 5.20]{geometricdual}}] \label{theoremgeometric} 
 	Let $X$ be a smooth connected Stein space over $K$ of dimension $d$. Assume $r,r'\geq 2d$ and $s=r+r'-d$. Then there is a natural quasi-isomorphism of complexes of $G_K$-equivariant sheaves
 	\begin{equation}
 		\cE_{\syn}(X_C,r)\cong R\sH om_{\FF}(\cE_{\syn,c}(X_C,r')[2d],\cO\otimes \bQ_p(s)).
 	\end{equation}
 \end{theorem}

\subsection{Local Tate duality for the syntomic sheaf}
$ $

In this section, we prove the local Tate duality for the compactly supported syntomic sheaf. 
 
 Let $\cM_r=(M,\rho_M,\varphi_M)\in D\nuc(B)^{\varphi,G_K}$ such that $R\underline{\Hom}_{B_{[u,v]}}(M,B_{[u,v]})\in D\nuc(B_{[u,v],\square})$ (we only define internal Hom under such an assumption). There is a natural map 
 $$\begin{aligned}
 R\sH om_{\FF}(\cM_r, \cO\otimes \bQ_p(1)[2]) \ &\otimes_{\FF}^L  R\sH om_{\FF}(\cO,\cM_r)\ \\
 & \to  R\sH om_{\FF}(\cO,\cO\otimes\ \bQ_p(1)[2])=\cO\otimes \bQ_p(1)[2].
 \end{aligned}$$
 Taking cohomology on the curve and Galois cohomology, we get a pairing 
$$R\Gamma(G_K,\bD_{\FF}(\cM_r))\otimes_{\bQ_p}^{L_\square} R\Gamma(G_K,R\Gamma(\FF,\cM_r))\to R\Gamma(G_K,\bQ_p(1))[2]\stackrel{\tau^{\geq 0}}\longrightarrow \bQ_p$$
where $\bD_{\FF}(\cM_r):=R\Hom_{\FF}(\cM_r,\cO_{\FF}\otimes \bQ_p(1)[2])\in D(\bQ_{p,\square}[G_K])$. Recall that we used $R\Hom_{\FF}(-,-)$  to denote the functor $R\Gamma(\FF,R\sH om(-,-))$.

It induces a morphism
\begin{equation}\label{formulapairing}
	\gamma_{\cM_r}:R\Gamma(G_K,\bD_{\FF}(\cM_r))\longrightarrow \bD_{\bQ_p}(R\Gamma(G_K,R\Gamma(\FF,\cM_r))).
	\end{equation}
 where $\bD_{\bQ_p}(-):=R\Hom_{\bQ_p}(-,\bQ_p)$. The morphism $\gamma_{\cM_r}$ is functorial in $\cM_r$. \\

  Let $X$ be a smooth geometrically connected Stein space of dimension $d$ over $K$.
When $\cM_r=\cE_{\syn,c}(X_C,r)$ (use Lemma \ref{sheavesaregood}), we get from definition that $R\Gamma_{\syn,c}^B(X_C,\bQ_p(r))=R\Gamma(\FF,\cM_r)$ and the morphism in \eqref{formulapairing} will be denoted by 
 $$\gamma^r_{\syn}:R\Gamma(G_K,\bD_{\FF}(\cE_{\syn,c}(X_C,r)))\longrightarrow \bD_{\bQ_p}(R\Gamma(G_K,R\Gamma^B_{\syn,c}(X_C,\bQ_p(r)))).$$

\begin{theorem}\label{theoremlocaltatesheaf1}
 	For $r\geq d$, the morphism $\gamma^r_{\syn}$ is a quasi-isomorphism.
 	 \end{theorem} 
  The remaining of this subsection is devoted to prove this theorem.
   
\begin{remark}\label{remarkequivariant}
	Assume $\cE,\cF\in D\nuc(B)^{\varphi,G_K}$ and note that the action of $\varphi$ and $G_K$ commutes. It follows from the explicit complex computing $R\Gamma(G_K,-)$ (cf. Lemma \ref{lemmacondenseandcontinuous}) and $R\Gamma(\FF,-)$ (cf. Lemma \ref{lemmaglobalsection}) that we have a quasi-isomorphism
	$$R\Hom_{D\nuc(B)^{\varphi,G_K}}(\cO,\cE)\simeq R\Gamma(G_K,R\Gamma(\FF,\cE)).$$
	By adjunction, we have quasi-isomorphisms
	$$ \begin{aligned}
	R\Hom_{D\nuc(B)^{\varphi,G_K}}(\cF,\cE)&\simeq R\Hom_{D\nuc(B)^{\varphi,G_K}}(\cO,R\sH om(\cF,\cE))
	\\ &\simeq R\Gamma(G_K,R\Gamma(\FF,R\sH om(\cF,\cE))).
	 \end{aligned}$$
	 Thus, the duality could be understood as a duality of 'equivariant sheaves' on $\FF$: When $\cM_r=\cE_{\syn,c}(X_C,\bQ_p(r))$
	 $$R\Hom_{D\nuc(B)^{\varphi,G_K}}(\cM_r,\cO\otimes\bQ_p(1)[2])\simeq \bD_{\bQ_p}(R\Hom_{D\nuc(B)^{\varphi,G_K}}(\cO,\cM_r)).$$
	  It is stated as in Theorem \ref{theoremlocaltatesheaf1} for ease of computations and applications in arithmetic duality.
	\end{remark}
	
	Taking $\cM_r$ to be the complex $\cE_{\HK,c}(X_C,r)$ (resp. $\cE_{\dR,c}(X_C,r)$) in \eqref{formulapairing} (use Lemma \ref{sheavesaregood} again), we get morphisms 
	$$\gamma^r_{\HK}:R\Gamma(G_K,\bD_{\FF}(\cE_{\HK,c}(X_C,r)))\longrightarrow \bD_{\bQ_p}(R\Gamma(G_K,R\Gamma^B_{\HK,c}(X_C,r))).$$
	$$\gamma^r_{\dR}:R\Gamma(G_K,\bD_{\FF}(\cE_{\dR,c}(X_C,r)))\longrightarrow \bD_{\bQ_p}(R\Gamma(G_K,R\Gamma^B_{\dR,c}(X_C,r))).$$

\subsubsection{Duality for the Hyodo-Kato sheaf} $ $

   The category $\sC(G_K)$ of almost $C$-representations is (essentially) small, thus by \cite[Theorem 8.6.5(vi), Theorem 15.3.1(i)]{kashiwara2006categories}, one can define the derived functor $R\Hom_{\sC(G_K)}(M,-)$ via the Ind-category for a complex $M$ in the bounded (classical) derived category $D^b(\sC(G_K))$. Note that any morphism in $\sC(G_K)$ is strict as a morphism in $\cB(G_K)$. Therefore, a sequence in $\sC(G_K)$ is exact if and only if it is exact when we consider it as condensed vector spaces. Combine it with Theorem \ref{theoremalmostcpdual} (1) (2) and compatibility between condensed and continuous group cohomology (Lemma \ref{lemmacondenseandcontinuous}), we get the following lemma
   \begin{lemma}\label{lemmauseful}
   	Let $Y\in D^b(\sC(G_K))$, there is an isomorphism of finite dimensional condensed $\bQ_p$-vector spaces
   	$$R^i\Gamma(G_K,\underline{Y})\cong \underline{R^i\Hom_{\sC(G_K)}(\bQ_p, Y)}.$$
   \end{lemma}

   \begin{remark}
       The author does not know if the natural functor between classical derived categories $D^b(\sC(G_K))\to D(\bQ_{p,\square}[G_K])$, induced by inclusion, is fully faithful. Addressing this issue is subtle, for the first, the fully faithfulness will imply any extension of $C$ by $C$, in the category of Banach $\bQ_p$-vector spaces with continuous $G_K$-action, automatically admits a $B^+_{\dR}$-linear structure. For the second, we note that there is no nonzero projective object in $\sC(G_K)$. This is because, by Euler characteristic formula \eqref{eulercharacter}, a projective object $X$ must have height $0$. Since each higher extension group between $X$ and any $Y$ vanishes, \eqref{eulercharacter} again tells that $\Hom_{\sC(G_K)}(X,Y)=0$ for any $Y\in \sC(G_K)$, in particular for $Y=X$.
   \end{remark}

By Theorem \ref{theoremfontaineequivalence}, given $\cF,\cG\in D^b(\cM(G_K))$, we have the following isomorphisms:
\begin{equation}\label{formulaequivofderived}
\begin{aligned}
	&R^n\Hom_{\cM(G_K)}(\cF,\cG)\cong \Hom_{D^b(\cM(G_K))}(\cF,\cG[n])\\
	&\cong \Hom_{D^b(\sC(G_K))}(R\Gamma(\FFalg,\cF),R\Gamma(\FFalg,\cG)[n])\\
	&\cong R^n\Hom_{\sC(G_K)}(R\Gamma(\FFalg,\cF),R\Gamma(\FFalg,\cG))
\end{aligned}\end{equation} 

It provides the following isomorphisms (the first one is Lemma \ref{lemmauseful}):
\begin{equation}\label{formulaisomorphismglobalsection}
R^n\Gamma(G_K,\underline{R\Gamma(\FFalg,\cF)})\cong \underline{R^n\Hom_{\sC(G_K)}(\bQ_p,R\Gamma(\FFalg,\cF))}\cong \underline{R^n\Hom_{\cM(G_K)}(\cO,\cF)}
\end{equation}
Assume $\cF,\cG\in D^b(\cM(G_K))$ are perfect complexes, so they can be represented by flat and dualizable complexes and we get \begin{equation}\label{formulaisomorphismoncohomology}
R^n\Gamma(G_K,\underline{R\Hom_{\FF^{\alg}}(\cF,\cG)})\cong  R^n\Gamma(G_K,\underline{R\Gamma(\FFalg,\cF^\vee\otimes \cG)})\cong \underline{R^n\Hom_{\cM(G_K)}(\cF,\cG)}.
\end{equation}
where the second isomorphism uses \eqref{formulaisomorphismglobalsection}.

\begin{proposition} \label{propositionlocaltateHKsheaf}
 	Suppose $r\geq d$. Then $\gamma_{\HK}^r$ is a quasi-isomorphism. \end{proposition}
 \begin{proof}

 Let $\{X_n\}$ be a strictly increasing open cover of $X$ by adapted naive interiors of dagger affinoids. We first show the morphism
  $$\gamma^{r}_{\HK,n}: R\Gamma(G_K,\bD_{\FF}(\cE_{\HK,c}(X_{n,C},r)))\longrightarrow \bD_{\bQ_p}(R\Gamma(G_K,R\Gamma(\FF,\cE_{\HK,c}(X_{n,C},r)))) $$
 is a quasi-isomorphism. 
 For each $i\geq 0$, $H^i_{\HK,c}(X_{n,C})$ is a finite dimensional vector space over $\check{C}$ with action of $G_K,\varphi$ and $N$. So $\cE_{\HK,c}^i(X_{n,C},r):=H^i(\cE_{\HK,c}(X_{n,C},r))$ is a vector bundle of non-negative slopes on $\FF$ with $G_K$-action (cf. Remark \ref{remarkslope}). By canonical truncation, it is enough to show that the morphism
 \begin{equation}\label{formulapropositionfirstreduce}
 \gamma^{r}_{\HK,i,n}: R\Gamma(G_K,\bD_{\FF}(\cE^i_{\HK,c}(X_{n,C},r)))\longrightarrow \bD_{\bQ_p}(R\Gamma(G_K,R\Gamma(\FF,\cE^i_{\HK,c}(X_{n,C},r))) 	
 \end{equation}
 is a quasi-isomorphism.

Since condensed perfect complexes on the Fargues-Fontaine curve coincide with classical perfect complexes: If $\sF^{cl}$ is a classical vector bundle on $\FF$ with associated condensed vector bundle $\sF$, then $H^i(\FF,\sF^{cl})$ is a topological $\bQ_p$-vector space and
    $$\underline{H^i(\FF,\sF^{cl})}\cong H^i(\FF,\sF).$$  
    In what follows, we will not distinguish them.

  Using the construction in \cite[Sec. 10.3.2]{fargues2019courbes}, one can also associate to $H^i_{\HK,c}(X_{n,C})$ a $G_K$-equivariant vector bundle $\cE^{i,alg}_{\HK,c}(X_C,r)$ on $\FFalg$  whose analytification is $\cE^i_{HK,c}(X_{n,C},r)$. By GAGA theorem (\cite[Proposition II 2.7]{fargues2021geometrization}), one sees easily that the cohomology groups of their dual coincide as Banach $G_K$-representations. Therefore, we have a quasi-isomorphism
  \begin{equation}\label{formulaGAGAforsections}
  \bD_{\FF}(\cE^i_{\HK,c}(X_{n,C},r))\simeq \underline{R\Hom_{\FFalg}(\cE^{i,alg}_{\HK,c}(X_{n,C},r),\cO\otimes \bQ_p(1)[2])} 
   \end{equation}
   and will omit the superscript $alg$ in what follows.\\
   
   Note that since $\cE^i_{\HK,c}(X_{n,C},r)$ is a vector bundle of non-negative slope (cf. Remark \ref{remarkslope}), we have a quasi-isomorphism $R\Gamma(\FFalg,\cE^{i}_{\HK,c}(X_{n,C},r)) \simeq H^0(\FFalg,\cE^{i}_{\HK,c}(X_{n,C},r))[0]$. Under this identification, we use $HK^i_{c,B}(X_{n,C},r)$ to denote $R\Gamma(\FFalg,\cE^{i}_{\HK,c}(X_{n,C},r))$ and view it as an object in $\sC^{\geq 0}(G_K)$. Then Theorem \ref{theoremalmostcpdual} (3) induces a duality of finite dimensional classical $\bQ_p$-vector spaces
  \begin{equation}\label{formulafontainetatedualsheaf}
  	R^k\Hom_{\sC(G_K)}(HK^i_{c,B}(X_{n,C},r),\bQ_p(1))\cong (R^{2-k}\Hom_{\sC(G_K)}(\bQ_p,HK^i_{c,B}(X_{n,C},r)))^\vee 
  \end{equation}

  To show \eqref{formulapropositionfirstreduce}, it suffices to pass to cohomology. For $k\in \bZ$,  we need to show that the morphism
  $$ \gamma_{\HK,i,n}^{r,k}: R^k\Gamma(G_K,\bD_{\FF}(\cE^i_{\HK,c}(X_{n,C},r)))\longrightarrow \bD_{\bQ_p}(R^{-k}\Gamma(G_K,HK^i_{c,B}(X_{n,C},r))) $$
  is an isomorphism. (We use here the fact that the Galois cohomology groups are finite dimensional over $\bQ_p$)
  
  But this follows from the following isomorphisms
 $$\begin{aligned}
 &R^k\Gamma(G_K,R\Hom_{\FF}(\cE^i_{\HK,c}(X_{n,C},r),\cO \otimes \bQ_p(1)[2]))\\
 &\cong R^{k+2}\Gamma(G_K, \underline{R\Hom_{\FFalg}(\cE^{i,alg}_{\HK,c}(X_{n,C},r),\cO \otimes \bQ_p(1))})\\
 &\cong \un{R^{k+2}\Hom_{\cM(G_K)}(\cE^{i,alg}_{\HK,c}(X_{n,C},r),\cO \otimes \bQ_p(1))}\\
 &\cong  \underline{R^{k+2}\Hom_{\sC(G_K)}(HK^i_{c,B}(X_{n,C},r),\bQ_p(1))}\\
 &\cong \bD_{\bQ_p}(\underline{R^{-k}\Hom_{\sC(G_K)}(\bQ_p,HK^i_{c,B}(X_{n,C},r))})\\
 &\cong  \bD_{\bQ_p}(R^{-k}\Gamma(G_K,HK^i_{c,B}(X_{n,C},r))).
 \end{aligned}$$

The first isomorphism is \eqref{formulaGAGAforsections}.
The second one is \eqref{formulaisomorphismoncohomology}. The third one is \eqref{formulaequivofderived}. The fourth one is \eqref{formulafontainetatedualsheaf} and the last one is \eqref{formulaisomorphismglobalsection}. The compatibility with the pairing we constructed is clear as they are defined via the same cup product under \eqref{formulaisomorphismoncohomology}.\par 
 Now we pass to limit: 
 Since $\cE_{\HK,c}(X_C,r)=\text{colim}_n \cE_{\HK,c}(X_{n,C},r)$, we have (use $\bQ_{p,\square}[G_K]$ is compact in the last quasi-isomorphism)
 $$\begin{aligned}
 &R\Gamma(G_K,\bD_{\FF}(\cE_{\HK,c}(X_C,r))\simeq R\Gamma(G_K,\bD_{\FF}(\colim_n \cE_{\HK,{c}}(X_{n,C},r)))\\
 &\simeq  R\lim_n R\Gamma(G_K,\bD_{\FF}(\cE_{\HK,{c}}(X_{n,C},r)))\simeq R\lim_n \bD_{\bQ_p}(R\Gamma(G_K,R\Gamma(\FF,\cE_{\HK,c}(X_{n,C},r)))\\ 
 & \simeq \bD_{\bQ_p}(R\Gamma(G_K,R\Gamma(\FF,\cE_{\HK,c}(X_C,r))).
 \end{aligned}$$
\end{proof}

\begin{remark}
	It is clear from the proof that if we twist $\cE_{\HK,c}(X_C,r)$ by a Tate twist, the duality still holds. That is, for any $a\in \bZ$, we have a quasi-isomorphism
	$$\gamma^r_{\HK}: R\Gamma(G_K,\bD_{\FF}(\cE_{\HK,c}(X_C,r)\otimes \bQ_p(a)))\longrightarrow  \bD_{\bQ_p}(R\Gamma(G_K,R\Gamma_{\HK,c}^B(X_C,r)\otimes \bQ_p(a)))$$
\end{remark}

	The proof of the proposition above also shows the following result
	\begin{lemma}
		Assume $\cF\in D^b(\cM(G_K))$ is a perfect complex, then we have a natural quasi-isomorphism
		$$R\Gamma(G_K,\bD_{\FF}(\cF))\stackrel{\simeq}\longrightarrow  \bD_{\bQ_p}(R\Gamma(G_K,R\Gamma(\FF,\cF))).$$
	\end{lemma}

\subsubsection{Duality for the de Rham sheaf}

\begin{proposition} \label{propositionDRlocaltatesheaf}
   For any $a\in \bZ$ and $r\geq d$,  the morphism $\gamma^r_{\dR}$ is a quasi-isomorphism. \end{proposition}
 
 \begin{proof}
 	 	
 	Let $\cE^i_{\dR,c}(X_C,r):=H^i(\cE_{\dR,c}(X_C,r))$.
 	By canonical truncations on $\cE_{\dR,c}(X_C,r)$ and induction, it suffices to show the morphism 
 	$$\gamma^{r}_{\dR,i}: R\Gamma(G_K,\bD_{\FF}(\cE^i_{\dR,{c}}(X_C,r)\otimes \bQ_p(a)))\longrightarrow \bD_{\bQ_p}(R\Gamma(G_K,R^i\Gamma_{\dR,c}^B(X_C,r)\otimes \bQ_p(a)))$$
 	is a quasi-isomorphism. Here we used the fact that $R^i\Gamma(\FF,\cE_{\dR,c}(X_C,r))\cong R\Gamma(\FF,\cE_{\dR,c}^i(X_C,r))$.  
 	
 	To simplify the notation, we let $B'=(B_{[u,v]},\bZ)_\square$ and view the sheaves just as $B'$-modules. This is harmless because the gluing data on $\cE_{\dR,c}(X_C,r)$ is trivial and the $G_K$-action only comes from the ring $B'$.\par 
 	 Suppose $V$ is a $K$-vector space of compact type and $k\geq 0$, then we have
 	\begin{equation}\label{computation1}\begin{aligned}
 	&R\underline{\Hom}_{B'}(V\otimes_K^{L_\square} B'/t^k, B')
 	\simeq R\underline{\Hom}_{K}(V, R\underline{\Hom}_{B'}(B'/t^k,B'))\\
 	&\simeq R\underline{\Hom}_{K}(V, (t^{-k}B'/B')[-1])
 	\simeq R\underline{\Hom}_{K}(V,K)\otimes^{L_\square}_K (t^{-k}B'/B')[-1]\\
 	&\simeq V^*\otimes^{\square}_K (t^{-k}B'/B')[-1]\\
 	 \end{aligned}\end{equation}
    where the third quasi-isomorphism follows from \cite[Theorem 3.40 (2)]{rodrigues2022solid} since $B'/t^kB'$ is Fr\'echet and the last one holds by Lemma \ref{lemmacompacttypedual}. \par 
    
    Set $s=d+r-1$. We have $\cE^i_{\dR,c}(X_C,r)=0$ except for $d\leq i\leq s$. In this case, we have an exact sequence of $B'$-modules
    $$0\to (H^d_c(X,\Omega^{i-d})/\im(d))\otimes^\square_K C(s-i)\to \cE^i_{\dR,c}(X_C,r)\to H^i_{\dR,c}(X)\otimes_K^{\square} (B'/t^{s-i}B')\to 0$$
    where we recall $H^d_c(X,\Omega^{i-d})$ and $H^i_{dR,c}(X)$ are $K$-vector spaces of compact type. It induces a fiber sequence
    $$ H^i_{\dR,c}(X)^*\otimes_K^\square (t^{i-s}B'/B')[1]\to \bD_{\FF}(\cE_{\dR,c}^i(X_C,r))\otimes \bQ_p(-1)\to (H^d_c(X,\Omega^{i-d})/\im(d))^*\otimes_K^\square C(i-s-1)[1]$$
    
  Set $N=1-a-s$, we apply generalized Tate's formula (\cite[Proposition 3.15]{curves}) to get
\begin{equation}\label{formulacheck1} \begin{aligned}
    R^k\Gamma(G_K,\bD_{\FF}(\cE^i_{\dR,c}&(X_C,r)\otimes \bQ_p(a)))\\
    &=
\begin{cases}
	(H^d_c(X,\Omega^{i-d})/\im(d))^*,\ &\text{if }  k=-1,0  \text{ and } i=1-N\\
	H^i_{\dR,c}(X)^*,\ &\text{if }  k=-1,0  \text{ and } i+N\leq 0< 1-a\\
	0, & \mathrm{others}. 
\end{cases}
\end{aligned} \end{equation}  

and 

\begin{equation}\label{formulacheck2}\begin{aligned}
H^k(\bD_{\bQ_p}(R\Gamma(G_K,& R^i\Gamma_{\dR,c}^B(X_C,r)\otimes \bQ_p(a))))\\ &=
\begin{cases}
	(H^d_c(X,\Omega^{i-d})/\im(d))^* , &\text{ if } k=-1,0 \text{ and } i=1-N\\
	H^i_{\dR,c}(X)^* , &\text{ if } k=-1,0 \text{ and } a\leq 0< 1-N-i\\
	0, &\ \mathrm{others}.
\end{cases} 
\end{aligned}
\end{equation}

Now it suffices to show the morphism in solid $\bQ_p$-vector spaces
\begin{equation}\label{formulachecktatedualsheaf}
 	\gamma_{\dR,i}^{r,k}:  R^k\Gamma(G_K,\bD_{\FF}(\cE^i_{\dR,{c}}(X_C,r)\otimes \bQ_p(a)))\longrightarrow H^k(\bD_{\bQ_p}(R\Gamma(G_K,R^i\Gamma_{\dR,c}^B(X_C,r)\otimes \bQ_p(a))))
 	 \end{equation}
 induces an isomorphism for all $k\in \bZ$. We assume $N,a,k$ satisfy the first two conditions in \eqref{formulacheck1} (or equivalently \eqref{formulacheck2}) since otherwise the groups are trivial.
We can take $\cF=i_{\infty*}C$ and $\cG=\cO_{\FF}\otimes \bQ_p(1)$ in \eqref{formulaisomorphismoncohomology} and we get
\begin{equation}\label{formulasome}
R^k\Gamma(G_K,R\Hom_{\FF}(i_{\infty*}C,\cO\otimes\bQ_p(1)))\cong R^k\Hom_{\sC(G_K)}(C,\bQ_p(1)) \end{equation}
 Note that $\cF$ is perfect because we have the exact sequence of equivariant coherent sheaves
$$ 0\to \cO_{\FF}\otimes \bQ_p(1) \to \cO_{\FF}(1)\to i_{\infty*}C\to 0.$$ \par 
  Meanwhile, for $l\in \bZ$, we have a duality of extension groups  
  $$\chi_l : \Ext^l_{\sC(G_K)}(C,\bQ_p(1))\cong   \Ext^{2-l}_{\sC(G_K)}(\bQ_p,C)^\vee.$$ 
   Since $gr^i(B_{\dR})\cong C(i)$ as $G_K$-representations and $R\Gamma(G_K,C(i))=0$ when $i\neq 0$, one can identify \eqref{formulachecktatedualsheaf}, after shifting degree, with the map
     $$ \psi_k: R^k\Gamma(G_K,V^*\otimes_K^\square C[-1])  \longrightarrow H^k(\bD_{\bQ_p}( R\Gamma(G_K,V\otimes_K^\square C[2]))),	
      $$
      where $V=(H^d_c(X,\Omega^{i-d})/\im(d))$ or $H^i_{\dR,c}(X)$.
      By \eqref{formulasome}, the computation \eqref{computation1} and the construction of the pairing, the morphism $\psi_k$ can be identified with $id_{V^*}\otimes^{\square} \chi_{k}$, which is an isomorphism. Note that we used the fact that both $V$ and $V^*$ are solidly nuclear spaces with trivial Galois action.
      
\end{proof}

\subsubsection{Duality for syntomic sheaf}

\begin{proof}[Proof of Theorem \ref{theoremlocaltatesheaf1}] 
 We have a commutative diagram of distinguished triangles
 \begin{equation}\label{diagramcommutes}
\xymatrix{
	R\Gamma(G_K,\bD_{\FF}(\cE_{\syn,c}(X_C,r)))\ar^{\gamma_{\syn}^r} [r] & \bD_{\bQ_p}(R\Gamma(G_K,R\Gamma_{\syn,c}(X_C,\bQ_p(r)))) \\
	R\Gamma(G_K,\bD_{\FF}(\cE_{\HK,c}(X_C,r)))\ar[u] \ar^{\gamma^r_{\HK}} [r] & \bD_{\bQ_p}(R\Gamma(G_K,R\Gamma^B_{\HK,c}(X_C,r)))\ar[u] \\
	R\Gamma(G_K,\bD_{\FF}(\cE_{\dR,c}(X_C,r)))\ar[u] \ar^{\gamma^r_{\dR}}[r]& \bD_{\bQ_p}(R\Gamma(G_K,R\Gamma^B_{\dR,c}(X_C,r)))\ar[u]  
}\end{equation}

By Proposition \ref{propositionlocaltateHKsheaf}, \ref{propositionDRlocaltatesheaf}, both $\gamma_{\HK}^r$ and $\gamma_{\dR}^r$ are quasi-isomorphisms, so is $\gamma^r_{\syn}$.

	\end{proof}

\subsection{Duality for pro-\'etale cohomology}\label{sectiontrace}
We first recall the definition of the geometric trace map (cf. \cite[Sec. 7.3]{achinger2025compactlysupportedpadicproetale}) and the construction of $\gamma_{\proet}$ in \eqref{formuladesired1}.

Let $X$ be a smooth geometrically irreducible partially proper variety over $K$ of dimension $d$. By the syntomic comparison theorem (Theorem \ref{theoremclassicalperiod}) we have a morphism of exact sequences
$$\xymatrix{
 & H^{2d}_{\proet,c}(X_C,\bQ_p(2d)) \ar@{-->}_{\tr(2d)}[d] \ar[r]& (H^{2d}_{\HK,c}(X_C)\otimes^\square_{\check{C}} \mathbf{B}^+_{\st})^{\varphi=p^{2d},N=0} \ar_{\simeq}[d] \ar[r]& H^{2d}_{\dR,c}(X)\otimes^\square_K \mathbf{B}^+_{\dR}/F^d \ar^{\simeq}[d] \\
0 \ar[r]& \bQ_p(d) \ar[r]& \mathbf{B}^{+,\varphi=p^d}_{\crys} \ar[r]& \mathbf{B}_{\dR}^+/F^d
}
$$
where the second row is the fundamental exact sequence in $p$-adic Hodge theory and the middle and right vertical maps are isomorphisms. Thus it induces the dotted arrow which is the geometric trace map $\tr(2d)$. This construction is compatible with open immersions of smooth partially proper varieties (cf. \cite[Lemma 7.16(2)]{achinger2025compactlysupportedpadicproetale}).

For each $r\in \bZ$, by Tate twist, we get $tr(r): H^{2d}_{\proet,c}(X_C,\bQ_p(r))\to \bQ_p(r-d)$. In particular, take $r=d+1$, we get the following map:
$$\tr: H^{2d+2}_{\proet,c}(X,\bQ_p(d+1))\stackrel{\simeq}\to H^2(G_K,H^{2d}_{\proet,c}(X_C,\bQ_p(d+1)))\to H^2(G_K,\bQ_p(1))\stackrel{\simeq}\to \bQ_p$$
which is the arithmetic trace map.

In fact, the arithmetic trace map $\tr$ is an isomorphism. It suffices to show the second arrow in the definition of $\tr$ is an isomorphism: We have an exact sequence 
\begin{equation}\label{formulatobetwisted}
0\to V\to H^{2d}_{\proet,c}(X_C,\bQ_p(2d))\stackrel{\tr(2d)}\to \bQ_p(d)\to 0\end{equation}
where $V:=\coker(H^{2d-1}(HK_c(X_C,2d))\to H^{2d-1}(DR_c(X_C,2d)))$. Twist \eqref{formulatobetwisted} by $\bQ_p(1-d)$ and take Galois cohomology. Since $H^{2d-1}(DR_c(X_C,2d))$ is a (finite) colimit of spaces of the form $W\otimes B_{\dR}^+/F^s$ where $W$ is a solidly nuclear $K$-vector space with trivial action, generalized Tate's formula (cf. \cite[Proposition 3.15]{curves}) shows that $H^2(G_K,H^{2d-1}(DR_c(X_C,2d)))\otimes \bQ_p(1-d))=0$ thus $H^2(G_K,V(1-d))=0$. This implies the claim.

So we get a pairing
$$\begin{aligned}
R\Gamma_{\proet}(X,\bQ_p(j)) &\otimes^{L_\square}_{\bQ_p} R\Gamma_{\proet,c}(X,\bQ_p(d+1-j))[2d+2]\\
&\to R\Gamma_{\proet,c}(X,\bQ_p(d+1))[2d+2]\stackrel{\tau^\geq 0}\to H^{2d+2}_{\proet,c}(X,\bQ_p(d+1))\stackrel{\tr}\longrightarrow \bQ_p
\end{aligned}$$
which induces a map
$$\gamma_{\proet}: R\Gamma_{\proet}(X,\bQ_p(j)) \to \bD_{\bQ_p}(R\Gamma_{\proet,c}(X,\bQ_p(d+1-j))[2d+2]).$$
Since the construction of geometric trace map is compatible with open immersion of smooth partially proper varieties, so is $\gamma_{\proet}$.\\

To handle the topology on the compactly supported pro-\'etale cohomology, we need the following lemma.
\begin{lemma}\label{lemmaHausdorff}
    Assume $\{V_n\}_{n\in \bN}$ is an inductive system of spaces of compact type with compact transition maps such that the projective system of the strong duals $\{V_n^*\}_{n\in \bN}$ satisfies topological Mittag-Leffler condition. Then $V=\colim_n V_n$ is a space of compact type.
\end{lemma}

\begin{proof}
     Let $\{V_n^*\}_n$ be the dual system of $\{V_n\}_n$. By assumption, for each $n\in \bN$ and sufficiently large $m$, $W_n:=\overline{\im(f_{m,n}:V_m^*\to V_n^*)}\subseteq V_n^*$ is a well-defined subspace of $V_n^*$ independent of $m$. The transition maps between $V_n^*$ induce transition maps between $W_n$ such that $\{W_n\}_{n\in \bN}$ forms a projective system of nuclear Fr\'echet spaces with dense transition maps. For sufficiently large $m$, the map $V_m^*\to V_n^*$ factors through $W_n$. This implies projective systems $\{V_n^*\}_{n\in \bN}$ and $\{W_n\}_{n\in \bN}$ are equivalent. Taking strong dual, we get equivalence between  inductive systems $\{V_n\}_{n\in \bN}$ and  $\{W_n^*\}_{n\in \bN}$ (here we used $V_n$ are reflexive). Since $W_n$ is a closed subspace of $V_n^*$ we have a strict exact sequence of nuclear Fr\'echet spaces
     $$ 0\to W_n\to V_n^*\to (V_n^*/W_n)\to 0.$$
     By \cite[Corollary 1.4]{schneider2002locally}, taking strong dual of the exact sequence above provides a strict exact sequence of spaces of compact type
     $$0\to (V_n^*/W_n)^*\to V_n\to W_n^*\to 0.$$
     By \cite[Theorem 8.1.3(xi)]{perez2010locally}, the morphisms between $W_n^*$ are compact. So $\{W_n^*\}_{n\in \bN}$ is an inductive system of spaces of compact type with compact injective transition map and it is equivalent to $\{V_n\}_{n\in \bN}$. We get $\colim_n V_n=\colim_n W_n^*$ is a space of compact type.
     
     \end{proof}

Now we are ready to deduce the main result of the paper.

 \begin{theorem}\label{theoremsecondmain}
	Let $X$ be a smooth Stein variety over $K$, geometrically irreducible of dimension $d$. Then:
	\begin{enumerate}
		\item The cohomology groups $H^i_{\proet}(X,\bQ_p(j))$ and $H^i_{\proet,c}(X,\bQ_p(j))$ are nuclear Fr\'echet and of compact type, respectively.
		\item The duality morphism 
		\begin{equation}\label{conjectureformula1} \gamma_{\proet}: 
		R\Gamma_{\proet}(X,\bQ_p(j))\longrightarrow \bD_{\bQ_p}(R\Gamma_{\proet,c}(X,\bQ_p(d+1-j))[2d+2])\end{equation}
		is a quasi-isomorphism in $D(\bQ_{p,\square})$, where $\bD_{\bQ_p}(-):=R\Hom_{\bQ_p}(-,\bQ_p)$ is the (derived) dual in $D(\bQ_{p,\square})$. Moreover, it induces isomorphisms of solid $\bQ_p$-vector spaces
		\begin{equation} \label{conjectureformula2}
		H^i_{\proet}(X,\bQ_p(j))\simeq H^{2d+2-i}_{\proet,c}(X,\bQ_p(d+1-j))^*,
		\end{equation}
		\begin{equation} \label{conjectureformula3}
		H^i_{\proet,c}(X,\bQ_p(j))\simeq H^{2d+2-i}_{\proet}(X,\bQ_p(d+1-j))^*.
		\end{equation}
		
	\end{enumerate}
\end{theorem}
 
\begin{proof}

Theorem \ref{theorem1} proves that cohomology groups $H^i_{\proet}(X,\bQ_p(j))$ are nuclear Fr\'echet.  

To prove the derived duality \eqref{conjectureformula1}, we choose $r,r'\geq 2d$ and $a\in \bZ$ such that $r+a=j$. In this case, since the (compactly supported) syntomic complex is concentrated in degrees $[0,2d]$, Proposition \ref{propositionperiodisomorphism} induces  quasi-isomorphisms 
	\begin{equation}\label{formulasyntoetale}
	R\Gamma_{\syn}^B(X_C,\bQ_p(r))\simeq R\Gamma_{\proet}(X_C,\bQ_p(r)),\end{equation}
	 \begin{equation}\label{formulacomsuppsyntoet}
	 	R\Gamma_{\syn,c}^B(X_C,\bQ_p(r'))\simeq R\Gamma_{\proet,c}(X_C,\bQ_p(r')).\end{equation}
 Now, formula \eqref{conjectureformula1} follows from the computations (recall $s=r+r'-d$)
 \begin{equation}\label{lastcomputation}
 \begin{aligned}
	R\Gamma_{\proet}(X,\bQ_p(j))&\simeq R\Gamma(G_K,R\Gamma_{\syn}^B(X_C,\bQ_p(r))\otimes\bQ_p(a))   \\
	&\simeq R\Gamma(G_K,R\Gamma(\FF,\cE_{\syn}(X_C,r))\otimes\bQ_p(a))\\
	&\simeq R\Gamma(G_K,R\Hom_{\FF}(\cE_{\syn,c}(X_C,r')[2d],\cO\otimes \bQ_p(s))\otimes\bQ_p(a))\\
	&\simeq \bD_{\bQ_p}(R\Gamma(G_K,R\Gamma_{\syn,c}^B(X_C,\bQ_p(r'))\otimes\bQ_p(-a-s+1))[2d+2])\\
	&\simeq \bD_{\bQ_p}(R\Gamma(G_K,R\Gamma_{\proet,c}(X_C,\bQ_p(d+1-j))[2d+2]))\\
	&\simeq \bD_{\bQ_p}(R\Gamma_{\proet,c}(X,\bQ_p(d+1-j))[2d+2]),
\end{aligned}\end{equation}
where the first and the fifth quasi-isomorphisms use \eqref{formulasyntoetale} and \eqref{formulacomsuppsyntoet}, respectively; the third one is Theorem \ref{theoremgeometric} and the fourth one is Theorem \ref{theoremlocaltatesheaf1}.\par 
    
    To check the quasi-isomorphism \eqref{lastcomputation} coincides with $\gamma_{\proet}$ (up to a non-zero constant) we recall that for an integer $r\geq 2d$,  there is a trace map for the compactly supported syntomic sheaf
$$\tr^{B}_{\syn}(r): \cE_{\syn,c}(X_C,\bQ_p(r))\longrightarrow \cO\otimes \bQ_p(r-d)[-2d]$$
defined in \cite[Sec. 5.2]{geometricdual}. Taking global section on the Fargues-Fontaine curve and the $2d$-th cohomology, $tr^{B}_{\syn}(r)$ induces a map
$$tr_{\syn}^{FF}(r): H^{2d}_{\syn,c}(X_C,\bQ_p(r))\to \bQ_p(r-d).$$
Apply the period isomorphism, we get a trace map 
$$tr^{FF}(r): H^{2d}_{\proet,c}(X_C,\bQ_p(r))\to \bQ_p(r-d).$$
By construction, the trace map $tr(2d)$ coincides with $tr^{FF}(2d)$.

 Now we simplify the notation: put $G(-):=R\Gamma(G_K,-)$, $F(-):=R\Gamma(\FF,-)$, $\cE_*(r):=\cE_{\syn,*}(X_C,\bQ_p(r))$, where $*\in \{\emptyset,c\}$. Recall that the arithmetic trace map \eqref{formulatrace} is the composition of the geometric trace map $tr(d+1)$ and the Galois trace map. Lemma \ref{lemmacompatible} implies that for any $r\geq 2d$ and $a\in \bZ$, the trace map 
$$tr^{FF}(r)\otimes\bQ_p(a): H^{2d}_{\proet,c}(X_C,\bQ_p(r))\otimes \bQ_p(a)\to \bQ_p(r-d+a)$$ 
coincides with $tr(r+a)$ (up to a nonzero constant).
    
    Choose $r,r'\geq 2d$, the pairing \eqref{pairingforarithmetic} now can be written as     
    \begin{equation*}\label{check}
    \begin{aligned}
    &G(F(\cE(r))\otimes \bQ_p(j-r))\otimes G(F(\cE_c(r'))\otimes \bQ_p(d+1-j-r'))\longrightarrow\\
    &G(F(\cE_c(r+r'))\otimes\bQ_p(d+1-r-r'))\longrightarrow G(F(\cO\otimes\bQ_p(1)[-2d]))\longrightarrow\\
    &G(\bQ_p(1))[-2d]\longrightarrow \bQ_p[-2d-2],
    \end{aligned} \end{equation*}
    where the first map is induced by cup product of syntomic sheaves, the second one is induced by $\tr^{FF}_{\syn}(r+r')$. Comparing it with the construction of $\gamma_{\syn}^r$, it remains to observe that, after the identification of $\cE(r)$ with $R\sH om(\cE_c(r'),\cO\otimes \bQ_p(s))$ which we used in the third quasi-isomorphism in \eqref{lastcomputation}, the pairing $\cE(r)\otimes \cE_c(r')\to \cO\otimes \bQ_p(s)$ is identified with the evaluation map $\cE_c(r')\otimes R\sH om(\cE_c(r'),\cO\otimes \bQ_p(s))\to \cO\otimes \bQ_p(s)$.

    Now we prove $H^i_{\proet,c}(X,\bQ_p(j))$ are of compact type. Let $\{X_n\}_{n\in \bN}$ be a strictly increasing open covering of X by adapted naive interiors of dagger affinoids. Since we already proved \eqref{conjectureformula1}, apply it to $X_n$, we get a quasi-isomorphism in $D(\bQ_{p,\square}):$
    $$\gamma_{\proet,n}:R\Gamma_{\proet}(X_n,\bQ_p(j))\stackrel{\simeq}\longrightarrow \bD_{\bQ_p}(R\Gamma_{\proet,c}(X_n,\bQ_p(d+1-j)[2d+2])).$$
    Taking cohomology, and since we have already shown $H^i_{\proet,c}(X_n,\bQ_p(j))$ are spaces of compact type (Theorem \ref{theorem2}), there are no higher extension groups arising in this duality (Lemma \ref{lemmacompacttypedual}). We get an isomorphism
    $$H^i_{\proet}(X_n,\bQ_p(j))\simeq H^{2d+2-i}_{\proet,c}(X_n,\bQ_p(d+1-j))^*.$$
    By Lemma \ref{lemmaHausdorff} and Theorem \ref{theorem2}, to show $H^i_{\proet,c}(X,\bQ_p(j))$ is of compact type, it suffices to show the system $\{H^i_{\proet}(X_n,\bQ_p(j))\}_{n\in \bN}$ satisfies topological Mittag-Leffler condition. This is proven in Theorem \ref{theorem1}(2).

    It's left to show \eqref{conjectureformula2} and \eqref{conjectureformula3}. The formula \eqref{conjectureformula2} follows from \eqref{conjectureformula1} by taking cohomology group and using $H^i_{\proet,c}(X,\bQ_p(j))$ are spaces of compact type (and Lemma \ref{lemmacompacttypedual}). Formula \eqref{conjectureformula3} follows from formula \eqref{conjectureformula2} (change $i$ to $2d+2-i$ and $j$ to $d+1-j$) and the fact that nuclear Fr\'echet spaces and spaces of compact type are reflexive. \par 
	
\end{proof}

\begin{lemma}\label{lemmacompatible}
        Let $r\geq 2d$, the following diagram commutes (up to a nonzero constant)
        $$\begin{CD}
           H^{2d}_{\proet,c}(X_C,\bQ_p(r))@>tr(r)>> \bQ_p(r-d)\\
            @V=VV @VV=V\\
            H^{2d}_{\proet,c}(X_C,\bQ_p(r))@>tr^{FF}(r)>> \bQ_p(r-d).
        \end{CD}$$
    \end{lemma}

    \begin{proof}
        Let $\epsilon$ be the generator of $\bZ_p(1)$ corresponding to $t$ and let $\phi$ be the composition of the isomorphisms
     \begin{equation*}
        \begin{tikzcd}
        H^{2d}_{\syn,c}(X_C,2d)\otimes \bQ_p(r-2d) \arrow[d, "\beta_{\syn}"] \arrow[r, dashed, "\phi"] & H^{2d}_{\syn,c}(X_C,r) \\
        H^{2d}_{\proet,c}(X_C,\bQ_p(2d))\otimes \bQ_p(r-2d) \arrow[r, "\cong"] & H^{2d}_{\proet,c}(X_C,\bQ_p(r)) \arrow[u, "\beta_{\syn}^{-1}"]
      \end{tikzcd}
     \end{equation*} 
Then $\phi(x\otimes \epsilon^{r-2d})=xt^{r-2d}$ (cf. \cite[Proof of Theorem 6.4]{colmez2024cohomologypadicanalyticspaces}). By the construction of the trace map, to show the Lemma, it suffices to show $\phi$ induces an isomorphism between the kernel of $tr(r)$ and $tr^{FF}(r)$. Note that when $r=2d$ the lemma follows directly from the construction. We therefore assume $r\geq 2d+1$. We have the following commutative diagram with exact rows
$$\begin{CD}
    0\to H^{2d-1}(HK_c(X_C,2d))@>t^{r-2d}>> H^{2d-1}(HK_c(X_C,r))@>>> H^{2d-1}_{\dR,c}(X)\otimes^{\square}_K B^+_{\dR}/F^{r-2d}\to  0\\
    @V\iota_{\HK,c}(2d)VV @VV\iota_{\HK,c}(r)V @VV=V \\
    0\to H^{2d-1}(DR_c(X_C,2d))@>t^{r-2d}>> H^{2d-1}(DR_c(X_C,r))@>>> H^{2d-1}_{\dR,c}(X)\otimes^{\square}_K B^+_{\dR}/F^{r-2d}\to  0\\
     @VVV @VVV \\
     H^{2d}_{\syn,c}(X_C,\bQ_p(2d)) @>t^{r-2d}>> H^{2d}_{\syn,c}(X_C,\bQ_p(r))
\end{CD}$$
where the first row is exact by taking $l=2d-1, j=r-1, k=r-2d-1$ in \cite[Lemma 8.1]{achinger2025compactlysupportedpadicproetale} (so $j-k=2d$ and the slope condition in the cited lemma is automatic). The claim follows because $\ker(tr(r))\cong\coker(\iota_{\HK,c}(2d))\otimes \bQ_p(r-2d)$, $\ker(tr^{FF}(r)) \cong \coker(\iota_{\HK,c}(r))$ and it is clear that the maps coincide with the one induced by $\phi$.
    \end{proof}

\begin{corollary}
	Let $X$ be a geometrically connected smooth partially proper rigid analytic variety over $K$ of dimension $d$. Then \eqref{formuladesired1} is a quasi-isomorphism in $D(\bQ_{p,\square})$
	$$\gamma_{\proet}: R\Gamma_{\proet}(X,\bQ_p(j))\stackrel{\simeq}\longrightarrow \bD_{\bQ_p}(R\Gamma_{\proet,c}(X,\bQ_p(d+1-j))[2d+2]).$$
\end{corollary}

\begin{proof}
	Since $X$ is partially proper (it is separated and countable at infinity by our convention), it admits an admissible covering $\{Y_i\}_{i\in \bN}$ by smooth Stein spaces $Y_i$. We set $Y=\coprod_{i} Y_i$. Since $X$ is separated, the intersection of two Stein spaces in $X$ is still a Stein space. By analytic (co-)descent of (compactly supported) pro-\'etale cohomology (cf. \cite[Sec. 2.2.4]{achinger2025compactlysupportedpadicproetale}), we have equivalences in $D(\bQ_{p,\square})$
	$$R\Gamma_{\proet}(X,\bQ_p(j))\simeq  R\lim_{[n]\in \Delta} R\Gamma_{\proet}(Y^{n/X},\bQ_p(j))$$
	$$R\Gamma_{\proet,c}(X,\bQ_p(j))\simeq  L\colim_{[n]\in \Delta} R\Gamma_{\proet,c}(Y^{n/X},\bQ_p(j)).$$
	where the right hand side of the quasi-isomorphisms are \v{C}ech nerve associated to the cover $Y\to X$ and $Y^{n/X}$ is the $n$-fold fiber product of $Y$ over $X$.
	
	 The quasi-isomorphism \eqref{conjectureformula1} is natural with respect to open inclusions of Stein spaces (because comparison theorems, Theorem \ref{theoremgeometric} and Theorem \ref{theoremlocaltatesheaf1} are functorial in this case). Thus
	 $$\begin{aligned}
	&\bD_{\bQ_p}(R\Gamma_{\proet,c}(X,\bQ_p(d+1-j))[2d+2])\\
	& \simeq \bD_{\bQ_p}(L\colim_{[n]\in \Delta} R\Gamma_{\proet,c}(Y^{n/X},\bQ_p(d+1-j))[2d+2]) \\
	&\simeq R\lim_{[n]\in \Delta} \bD_{\bQ_p}(R\Gamma_{\proet,c}(Y^{n/X},\bQ_p(d+1-j))[2d+2])\\
	&\simeq R\lim_{[n]\in \Delta}  R\Gamma_{\proet}(Y^{n/X},\bQ_p(j))\simeq R\Gamma_{\proet}(X,\bQ_p(j)),
\end{aligned}$$
as wanted.
\end{proof}

We have an analogous theorem for smooth dagger affinoids.

\begin{corollary}\label{corollarydaggeraffinoid}
	Let $X$ be a geometrically connected, smooth dagger affinoid variety over $K$ of dimension $d$. Assume $\{X_n\}$ is a dagger presentation of $X$ and $X_n^0$ is a naive interior of $X_n$ adapted to $\{X_n\}$. Then 
	
	\begin{enumerate}
		\item The projective system $\{ H^i_{\proet,c}(X_n^0,\bQ_p(j)) \}_{n\in \bN}$ satisfies topological Mittag-Leffler condition.
		\item We have isomorphisms 
		$$H^i_{\proet,c}(X,\bQ_p(j))\cong \lim_n H^i_{\proet,c}(X_n^0,\bQ_p(j))$$
		 and $H^i_{\proet,c}(X,\bQ_p(j))$ is a nuclear Fr\'echet space.
	
	\item We have isomorphisms of solid $\bQ_p$-vector spaces
	\begin{equation}\label{formulacorollary1}
	H^i_{\proet}(X,\bQ_p(j))\cong H^{2d+2-i}_{\proet,c}(X,\bQ_p(d+1-j))^*	,
	\end{equation}
	\begin{equation}\label{formulacorollary2}
	H^i_{\proet,c}(X,\bQ_p(j))\cong H^{2d+2-i}_{\proet}(X,\bQ_p(d+1-j))^*.
	\end{equation}
	
	\end{enumerate}
\end{corollary}

\begin{proof}
For any $n\in \bN$, to simplify the notation, we use $V_{n,c}$ to denote $H^i_{\proet,c}(X^0_{n},\bQ_p(j))$ and $V_n$ to denote $H^{2d+2-i}_{\proet}(X^0_n,\bQ_p(d+1-j))$. The correspond cohomology groups of $X$ are denoted by $V_c$ and $V$.
 
 For the first claim, let $m>n$ and we have a strict open inclusion $X_m^0\subseteq X_n^0$.  By the duality for $X_m^0$ and $X_n^0$ we have pairings
 $$\xymatrix{
	V_{m,c} \ar_{f^c_{m,n}}[d]& \times & V_m \ar[rr]& &\bQ_p\ar^{=}[d]\\
	V_{n,c} & \times & V_n \ar_{f_{m,n}}[u] \ar[rr]& &\bQ_p
	}$$
	such that for $x\in V_{m,c}$ and $y\in V_n$ we have $\langle f_{m,n}^c(x),y\rangle= \langle x,f_{m,n}(y)\rangle$. In fact, we have $\overline{\im(f_{m,n}^c)}=\ker(f_{m,n})^{\perp}$. One direction of inclusion is clear. To show $\ker(f_{m,n})^{\perp}\subseteq \overline{\im(f_{m,n}^c)}$, by perfectness of the pairings, we have $\overline{\im(f_{m,n}^c)}^{\perp}=\im(f_{m,n}^c)^{\perp}\subseteq \ker(f_{m,n})$. Taking perpendicular space, we get $\ker(f_{m,n})\subseteq \overline{\im(f_{m,n}^c)}^{\perp\perp}$. Since $V_{n,c}$ is a space of compact type, by Hahn-Banach theorem, we have $\overline{\im(f_{m,n}^c)}^{\perp\perp}=\overline{\im(f_{m,n}^c)}$. 
	So, to show $\{\overline{\im(f_{m,n}^c)}\}_m$ is stable, it is enough to consider $\{\ker(f_{m,n})\}_m$. Let $K:=\ker(V_n\to V)$. Since we already know $V$ is a space of compact type (Corollary \ref{corollarytopologydagger}), we know $K$ is a closed subspace of $V_n$ thus is Fr\'echet. On the other hand, we have $K=\cup_{m\geq n} \ker(f_{m,n})$. Since $\ker(f_{m,n})$ are closed subspaces of $K$, by Baire's category theorem applied to $K$, we get some $m_0$ such that $K=\ker(f_{m_0,n})$ which gives us the conclusion. 

	For the second claim, first note that 
     $$R\Gamma_{\proet,c}(X,\bQ_p(j))\simeq R\lim_n R\Gamma_{\proet,c}(X_n^0,\bQ_p(j)).$$
     This is because of the following computation:
$$\begin{aligned}
    R\Gamma_{\proet,c}(X,\bQ_p(j))&\simeq [R\Gamma_{\proet}(X_0,\bQ_p(j))\to R\Gamma_{\proet}(X_0\setminus\widehat{X},\bQ_p(j))]\\
    &\simeq [R\Gamma_{\proet}(X_0,\bQ_p(j))\to R\lim_n R\Gamma_{\proet}(X_0\setminus X_n,\bQ_p(j))]\\
    &\simeq R\lim_n R\Gamma_{\proet,c}(X_n^\dagger,\bQ_p(j))\simeq R\lim_n R\Gamma_{\proet,c}(X_n^0,\bQ_p(j))
\end{aligned}$$
where $X_n^\dagger$ is the dagger structure endowed on $X_n$ using the strict inclusion $X_n\Subset X_0$. Here in the first and the third quasi-isomorphisms, we use \cite[Proposition 2.37]{achinger2025compactlysupportedpadicproetale}; the second one follows from analytic descent and the fourth one is because for each $n$ we have inclusion of dagger varieties
$$X_n^0\subset X_n\subset X_{n-1}^0\subset X_{n-1}$$
and compatibility between dagger and rigid compactly supported cohomology  (\cite[Lemma 2.35]{achinger2025compactlysupportedpadicproetale}). So it suffices to show $R^1\lim_n V_{n,c}=0$. Since the system $\{V_{n,c}\}_{n\in \bN}$ satisfies topological Mittag-Leffler condition, this system is equivalent to $\{W_{n,c}\}_{n\in \bN}$ where $W_{n,c}:=\overline{\im(f_{m,n}^c)}$ for $m$ sufficiently large (for example, see the argument in the proof of Lemma \ref{lemmaHausdorff}). Now $\{W_{n,c}\}_n$ is a system of spaces of compact type with compact dense transition images. By \cite[Discussion after Proposition 16.5]{schneider2013nonarchimedean}, the system $\{W_{n,c}\}_n$ is equivalent to a projective system $\{U_{n,c}\}_n$ of Banach spaces with compact dense transition maps. Thus $R^1\lim_n V_{n,c}=R^1\lim_n U_{n,c}=0$.

	 To show the third claim. By claim (2) and formula \eqref{conjectureformula3}, we have isomorphisms
	$$\begin{aligned}
	H^i_{\proet,c}(X,\bQ_p(j))&\cong \lim_n H^i_{\proet,c}(X_n^0,\bQ_p(j)) \\
	&\cong \lim_n H^{2d+2-i}_{\proet}(X_n^0,\bQ_p(d+1-j))^*\\
	&\cong (\mathrm{colim}_n H^{2d+2-i}_{\proet}(X_n^0,\bQ_p(d+1-j)))^*\\
	&\cong H^{2d+2-i}_{\proet}(X,\bQ_p(d+1-j))^*.
\end{aligned}$$

Formula \eqref{formulacorollary1} follows from \eqref{formulacorollary2} by changing $i$ to $2d+2-i$ and $j$ to $d+1-j$ and using the fact that $H^i_{\proet}(X,\bQ_p(j))$ and $H^{2d+2-i}_{\proet,c}(X,\bQ_p(d+1-j))$ are reflexive (Corollary \ref{corollarytopologydagger}).

\end{proof}

\bibliographystyle{plain}
\bibliography{Arithmetic_duality_version5}

\Addresses

\end{document}